\documentclass[11pt, oneside]{amsart}
\usepackage{fullpage}
\usepackage{cite}
\usepackage[utf8]{inputenc}
\usepackage[T1]{fontenc}
\usepackage[english]{babel}
\usepackage{amsmath,amsfonts,amsthm, mathrsfs}
\usepackage[linktocpage=true]{hyperref}
\usepackage{tikz-cd}
\usepackage{array}

\theoremstyle{plain}
\newtheorem{theorem}{Theorem}[section]
\newtheorem{corollary}[theorem]{Corollary}
\newtheorem{lemma}[theorem]{Lemma}
\newtheorem{proposition}[theorem]{Proposition}

\theoremstyle{definition}

\newtheorem{remark}[theorem]{Remark}
\numberwithin{equation}{section}
\numberwithin{figure}{section}
\numberwithin{table}{section}

\DeclareMathOperator{\PP}{\operatorname{\mathbb P}}

\newcommand{\C}[0]{\mathbb C}
\newcommand{\N}[0]{\mathbb N}

\usetikzlibrary{shapes.multipart}
\usetikzlibrary{patterns}
\usetikzlibrary{shapes.multipart}
\usetikzlibrary{arrows}
\usetikzlibrary{decorations.markings}
\usepgflibrary{decorations.shapes}
\usetikzlibrary{decorations.shapes}
\usepgflibrary{shapes.symbols}
\usetikzlibrary{shapes.symbols}
\usetikzlibrary{decorations.pathreplacing}

\tikzstyle{fleche}=[>=stealth', postaction={decorate}, thick]
\tikzstyle{axis}=[->, >=stealth', thick, gray]
\tikzstyle{path}=[->, >=stealth', thick]
\tikzstyle{grille}=[dotted, gray]
\colorlet{lgray}{white!85!black}
\colorlet{lred}{white!65!red}
\colorlet{lblue}{white!75!blue}

\title[]{Periodic $q$-Whittaker and Hall--Littlewood processes}

\author{Jimmy He}
\address{Department of Mathematics, MIT, Cambridge, MA 02139}
\email{jimmyhe@mit.edu}

\author{Michael Wheeler}
\address{School of Mathematics and Statistics, University of Melbourne, Victoria 3010, Australia}
\email{wheelerm@unimelb.edu.au}

\date{\today}

\begin{document}
\maketitle

\begin{abstract}
We study the periodic $q$-Whittaker and Hall--Littlewood processes, two probability measures on sequences of partitions. We prove that a certain observable of the periodic $q$-Whittaker process exhibits a $(q,u)$ symmetry after a random shift, generalizing a previous result of Imamura, Mucciconi, and Sasamoto who showed a matching between the periodic Schur and $q$-Whittaker measures, and also give a vertex model formulation of their result. As part of our proof of the $(q,u)$ symmetry, we obtain contour integral formulas for both the periodic $q$-Whittaker and Hall--Littlewood processes. We also show a matching between certain observables in the periodic Hall--Littlewood process and in a quasi-periodic stochastic six vertex model after a suitable random shift, and discuss a limit to the stationary periodic stochastic six vertex model.
\end{abstract}

\section{Introduction}
There has been considerable success in studying integrable models in the KPZ universality class via probability measures on partitions or sequences of partitions defined in terms of symmetric functions. These include the Macdonald process, introduced by Borodin and Corwin \cite{BC14}, as well as its degenerations, which are related to certain integrable probabilistic models without a boundary including models of particle systems, random growth, directed polymers, and the KPZ equation. These measures on partitions have been extremely useful in obtaining asymptotics in these models by allowing for the computation of exact formulas, see e.g. \cite{BC14,CD18,BCF14,BCS14}.

There has also been considerable success in studying models with one open boundary via the half space Macdonald process and its degenerations, introduced by Barraquand, Borodin, and Corwin \cite{BBC20}. However, while exact formulas for certain observables were found, until the recent work of Imamura, Mucciconi, and Sasamoto \cite{IMS21, IMS22}, asymptotic results were restricted to zero temperature models where Pfaffian formulas were available except in very restricted cases \cite{BBCW18}. Their work \cite{IMS21} found a relationship between certain degenerations of the Macdonald and half space Macdonald processes, known as $q$-Whittaker and half space $q$-Whittaker processes, and two other objects called the periodic and free boundary Schur processes respectively. 

These periodic and free boundary models are also of independent interest. The periodic Schur process was introduced by Borodin \cite{B07}, and the free boundary Schur process by Betea, Bouttier, Nejjar, and Vuletic \cite{BBNV18}. Let us also mention that the periodic Macdonald process was introduced by Koshida \cite{K21}, who found formulas for moments of certain observables. These processes should be related to probabilistic models with either periodic or two open boundaries, but this relationship is not as clearly understood as for the full and half space models, and currently connections are only known in the Schur case.

In this paper, we study two degenerations of the periodic Macdonald process: the periodic $q$-Whittaker process, and the periodic Hall--Littlewood process. Our main results follow two threads: understanding the matching between $q$-Whittaker and periodic Schur measures, and understanding the relationship between periodic measures on partitions and probabilistic models with periodicity.

On the one hand, we generalize a result of Imamura, Mucciconi, and Sasamoto \cite{IMS21} on the relationship between the $q$-Whittaker and periodic Schur processes, by showing that it lifts to a statement about distributional symmetries in the two parameters of certain observables of the periodic $q$-Whittaker process after a random shift. As part of the proof, we actually give explicit contour integral formulas for the distribution function of this observable after the random shift. We also give a vertex model interpretation of the special case proved in \cite{IMS21}. We hope that both results will help shed light on the nature of these unexpected relationships between different measures on partitions and potentially lead to new asymptotic results. 

We also study the periodic Hall--Littlewood process, showing that certain observables can be matched to observables of a quasi-periodic six vertex model, generalizing the known relationship between the usual six vertex model and the Hall--Littlewood process found by Borodin, Bufetov, and the second author \cite{BBW16}. We also explain the relationship to the periodic six vertex model started from stationarity. There has been recent interest in understanding models in the KPZ universality class with periodic boundary conditions including particle systems \cite{BL16,L18,BL21,BL19}, polymers \cite{GK23b}, and the KPZ equation \cite{DGK23, GK23}. We hope that our results will lead to applications in studying a periodic six vertex model.

We now more formally state our main results. Throughout this paper, we will use $u$ to denote the winding fugacity of our periodic processes, and $q$ and $t$ to denote the remaining parameter in the $q$-Whittaker and Hall--Littlewood processes respectively, to match the conventions for Macdonald polynomials.

\subsection{Symmetries of periodic \texorpdfstring{$q$}{q}-Whittaker measures}
Let $q,u\in [0,1)$, and let $a$ and $b$ be positive specializations of the $q$-Whittaker functions $P_{\lambda/\mu}(x;q,0)$. We define the \emph{periodic $q$-Whittaker measure} to be the following measure on partitions:
\begin{equation*}
    \PP(\lambda)=\frac{1}{\Phi(a,b;q,0,u)}\sum_{\mu} u^{|\mu|}P_{\lambda/\mu}(a;q,0)Q_{\lambda/\mu}(b;q,0),
\end{equation*}
where $\Phi(a,b;q,0,u)$ is a normalization constant and $Q_{\lambda/\mu}(x;q,0)$ denotes the dual $q$-Whittaker function. If $a=(a_1,\dotsc)$ and $b=(b_1,\dotsc)$, then $\Phi(a,b;q,0,u)=\frac{1}{(u;u)_\infty}\prod_{i,j}\frac{1}{(a_ib_j;q,u)_{\infty}}$, where $(z;q,u)_\infty=\prod_{k,l\geq 0}(1-q^ku^lz)$. We say that $\chi$ is a $q$-geometric random variable if $\PP(\chi=n)=q^n\frac{(q;q)_\infty}{(q;q)_n}$. We now give the first main result of this paper.

\begin{theorem}
\label{thm: qt sym}
Let $\lambda$ be distributed according to the periodic $q$-Whittaker measure, and let $\chi$ be an independent $q$-geometric random variable. Then $\PP(\lambda_1+\chi\leq n)$ is symmetric in $q$ and $u$ for all $n$.
\end{theorem}
This result is actually a corollary of Theorem \ref{thm: qt sym fn}, where an explicit contour integral formula is given which is clearly symmetric in $q$ and $u$. We also obtain formulas in the Hall--Littlewood case, see Corollary \ref{cor: HL meas formula}. Our proof uses a symmetric function argument rather than probabilistic tools. When $u=0$, this reduces to Theorem 10.11 of \cite{IMS21}. In this case, we are able to give another argument using vertex models. This proof has the benefit of being more conceptual, and we hope that it will lead to greater insight on the extent of these symmetries and identifications between different measures on partitions.

\subsection{Six vertex model and periodic Hall--Littlewood process}

Let $t,u\in [0,1)$, and let $a=(a_1,\dotsc, a_M)$ and $b=(b_1,\dotsc, b_N)$ be parameters in $(0,1)$. We define the \emph{periodic Hall--Littlewood process} to be the following probability measure on sequences of partitions $\vec\lambda=(\lambda^{(0)}\subseteq\dotsm\subseteq \lambda^{(M)})$ and $\vec\mu=(\mu^{(0)}\subseteq \dotsm\subseteq\mu^{(N)})$ with $\mu^{(0)}=\lambda^{(0)}$ and $\mu^{(N)}=\lambda^{(M)}$:
\begin{equation*}
    \PP(\vec\lambda,\vec\mu)=\frac{u^{|\mu^{(0)}|}}{\Phi(a,b;0,t,u)}\prod_{i=1}^MP_{\lambda^{(i)}/\lambda^{(i-1)}}(a_i;0,t)\prod_{j=1}^N Q_{\mu^{(j)}/\mu^{(j-1)}}(b_j;0,t),
\end{equation*}
where $\Phi(a,b;0,t,u)$ is a normalization constant, and $P_{\lambda/\mu}(x;0,t)$ and $Q_{\lambda/\mu}(x;0,t)$ denote the Hall--Littlewood polynomials and their duals.

We now informally define the quasi-periodic stochastic six vertex model, and refer the reader to Section \ref{sec: 6vm} for a formal definition. The usual six vertex model on an $M\times N$ domain is a probability distribution on arrows traveling up and right through the edges of an $M\times N$ lattice. Each edge contains at most one arrow, and probabilities are determined by a product of vertex weights. The model contains many parameters, including rapidities $a_i$ and $b_j$ associated to columns and rows respectively, as well as an anisotropy parameter $t$, and the vertex weight at $(i,j)$ depends on $a_i$, $b_j$, and $t$.

To define the quasi-periodic stochastic six vertex model, we introduce a new parameter $u\in [0,1)$. We then build the quasi-periodic model out of infinitely many $M\times N$ domains, where we use column rapidities $u^{k-1}a_i$ and row rapidities $b_j$ in the $k$th copy of the domain. The edges are connected so that the top of one domain enters the bottom of the next, and the right of one domain enters the left of the next. One can also view this as a model defined on a cylinder. The decay in rapidities ensures that it is possible to allow arrows to travel infinitely far. In the final domain, arrows exit, and we are interested in these locations. We also keep track of the number of times arrows wind in this model.


Our second main result, Theorem \ref{thm: HL 6vm}, states that certain observables in the periodic Hall--Littlewood process match those of the quasi-periodic six vertex model in distribution. In particular, the distribution of the length $l(\mu^{(0)})$ matches the number of times arrows wind shifted by an independent $u$-geometric random variable, and whether $l(\lambda^{(i)})>l(\lambda^{(i-1)})$ or not matches whether an arrow exits or not at column $i$ (a similar statement also holds for $\vec\mu$ and the rows). Our proof uses another integrable vertex model of deformed bosons and Yang--Baxter graphical arguments. Similar ideas were used in \cite{BBW16}, which proved the special case when $u=0$. 

The model we consider is not quite periodic due to the parameter $u$, and so we also consider the $u\to 1$ limit. We show that certain observables survive this limiting procedure and converge to those in a stationary periodic stochastic six vertex model (Proposition \ref{prop: stat 6vm}), which means that the periodic Hall--Littlewood process is also related to this model. Unfortunately, the most interesting statistic, the number of times arrows wind, diverges during this procedure, but we hope that one can still extract useful information about the periodic six vertex model from our matching of the distributions.

\subsection{Organization}
In Section \ref{sec: prelim}, we review some background on symmetric functions and periodic Macdonald processes. In Section \ref{sec: sym}, we establish the $(q,u)$ symmetry of periodic $q$-Whittaker measures given by Theorem \ref{thm: qt sym}. In Section \ref{sec: t=0} we give a vertex model proof when $u=0$. In Section \ref{sec: 6vm}, we match certain observables of the periodic Hall--Littlewood process with observables in a quasi-periodic six vertex model and explain the limit to the statinoary periodic six vertex model.

\subsection{Notation}
For a collection of variables $x=(x_1,\dotsc, x_N)$, we let $x^{-1}=(x_1^{-1},\dotsc, x_N^{-1})$. We let $(z;q)_n=(1-z)\dotsm(1-q^{n-1}z)$ with $n=\infty$ possible, and use $(z;q,t)_\infty=\prod_{k,l\geq 0}(1-q^kt^lz)$ to denote a two-parameter version of the $q$-Pochhammer symbol. 

\section{Preliminaries}
\label{sec: prelim}
\subsection{Macdonald polynomials}
We refer the reader to \cite{M79} for further background on symmetric functions and Macdonald polynomials.

A \emph{partition} is a finite non-increasing sequence of non-negative integers $\lambda=(\lambda_1\geq \lambda_2\geq \dotsm\geq \lambda_k)$, and we call $k$ the \emph{length}, denoted $l(\lambda)$. We let $|\lambda|=\sum_{i\geq 1} \lambda_i$ denote its \emph{size}. It is frequently useful to view a partition as its \emph{Young diagram}, where we place $\lambda_i$ boxes in the $i$th row from top to bottom, so that each row is aligned on the left. We define the \emph{conjugate partition} $\lambda'$ as the partition obtained from $\lambda$ by reflecting its Young diagram, switching the rows and columns. Given a box $s\in\lambda$, we let $a(s)$ and $l(s)$ be the \emph{arm length} and \emph{leg length}, defined as the number of boxes to the right and below respectively. We will let $m_i(\lambda)$ denote the number of occurrences of $i$ in $\lambda$.

Let $x=(x_1,x_2,\dotsc)$ be a formal alphabet, and let $\Lambda$ denote the ring of symmetric functions. We let $P_\lambda(x;q,t)$ denote the \emph{Macdonald polynomials}, defined as the unique symmetric functions orthogonal with respect to the inner product defined by
\begin{equation*}
    \langle p_\lambda,p_\mu\rangle=\delta_{\lambda,\mu}z_\lambda\prod_i\frac{1-q^{\lambda_i}}{1-t^{\lambda_i}},
\end{equation*}
where $p_\lambda=\prod_{i\geq 1} p_{\lambda_i}$, $p_k=\sum_{i\geq 1}x_i^k$, and $z_\lambda=\prod _i m_i(\lambda)!i^{m_i(\lambda)}$, and whose change of basis to the monomial symmetric functions is upper triangular with respect to the dominance ordering on partitions (see Chapter VI, Section 4 of \cite{M79}). We let $Q_{\lambda}(x;q,t)$ denote the dual basis with respect to this inner product, and define $b_\lambda(q,t)$ by $Q_\lambda(x;q,t)=b_\lambda(q,t) P_\lambda(x;q,t)$. We have
\begin{equation*}
b_\lambda(q,t)=\prod_{s\in\lambda}\frac{1-q^{a(s)}t^{l(s)+1}}{1-q^{a(s)+1}t^{l(s)}},
\end{equation*}
and note that $b_\lambda(q,t)=b_{\lambda'}(t,q)$. Macdonald polynomials satisfy a \emph{Cauchy identity}, which states that
\begin{equation*}
    \sum_{\lambda}P_\lambda(x;q,t)Q_{\lambda}(y;q,t)=\prod_{i,j}\frac{(tx_iy_j;q)_\infty}{(x_iy_j;q)_\infty}=:\Pi(x,y;q,t).
\end{equation*}

For a skew Young diagram $\lambda/\mu$, the \emph{skew Macdonald polynomials} $P_{\lambda/\mu}(x;q,t)$ are then defined by
\begin{equation*}
    \langle P_{\lambda/\mu},Q_\nu\rangle=\langle P_\lambda, Q_\mu Q_\nu\rangle
\end{equation*}
for all $Q_\nu$. The $Q_{\lambda/\mu}$ are similarly defined, with the roles of $P$ and $Q$ swapped. These are homogeneous polynomials in the alphabet $x$ of degree $|\lambda|-|\mu|$. They satisfy a \emph{branching rule}, meaning that if we specialize into two sets of variables $(a,b)$, then
\begin{equation*}
    P_{\lambda/\mu}(a,b;q,t)=\sum_{\nu}P_{\lambda/\nu}(a;q,t)P_{\nu/\mu}(b;q,t),
\end{equation*}
and similarly for the dual family $Q_{\lambda/\mu}$.

Recall that $\Lambda$ is isomorphic to a polynomial ring in variables $p_k$, so any homomorphism from $\Lambda$ is uniquely defined by a choice of where to send $p_k$ for each $k$. There is an important involution $\omega_{q,t}$ on the ring of symmetric functions defined by
\begin{equation*}
    \omega_{q,t}(p_\lambda)=(-1)^{|\lambda|-l(\lambda)}\prod_i\frac{1-q^{\lambda_i}}{1-t^{\lambda_i}}p_\lambda.
\end{equation*}
This involution satisfies $\omega_{q,t}(P_{\lambda/\mu}(x;q,t))=Q_{\lambda'/\mu'}(x;t,q)$, and $\omega_{q,t}(Q_{\lambda/\mu}(x;q,t))=P_{\lambda'/\mu'}(x;t,q)$, simultaneously swapping the roles of $P$ and $Q$, along with $q$ and $t$, while also conjugating all partitions.

We will be interested in two special cases of the Macdonald polynomials. When $q=0$, the Macdonald polynomials are called the \emph{Hall--Littlewood polynomials}, and when $t=0$, they are called the \emph{$q$-Whittaker polynomials}. The common specialization $q=t=0$ (actually $q=t$ suffices) are called the \emph{Schur polynomials}, which we will denote $s_{\lambda}(x)$.

\subsection{Orthogonality and integral formulas}
We now give some contour integral formulas for the Hall--Littlewood polynomials which we will need. These could be given for general Macdonald polynomials, but certain constants become much more complicated. 

Specialize to $n$ variables $z=(z_1,\dotsc, z_n)$. Following Chapter VI, Section 9 of \cite{M79}, given two Laurent polynomials $f,g$ in $z$ (or more generally any formal series such that the product below is well-defined), we let
\begin{equation*}
    \langle f,g\rangle_n'=\frac{1}{n!}[f(z)g(z^{-1})\Delta(z;q,t)]_1,
\end{equation*}
where $[\cdot]_1$ means we take the constant term in $z$, and
\begin{equation*}
    \Delta(z;q,t)=\prod_{i\neq j}\frac{(z_iz_j^{-1};q)_\infty}{(tz_iz_j^{-1};q)_\infty}.
\end{equation*}
The Macdonald polynomials are orthogonal with respect to this inner product, cf. \cite[(9.5)]{M79}. Moreover, specializing $t=0$, we have
\begin{equation*}
    \langle P_\lambda(z;0,t),Q_\mu(z;0,t)\rangle_n'=\mathbf{1}_{\lambda=\mu}\mathbf{1}_{l(\lambda)\leq n}\frac{(1-t)^n}{(t;t)_{n-l(\lambda)}},
\end{equation*}
see e.g. Theorem 3.1 of \cite{V12}. If $f$ and $g$ are Laurent polynomials, then
\begin{equation*}
    \langle f,g\rangle_n'=\frac{1}{n!}\oint_{C} \frac{dz_1}{2\pi i z_1}\dotsm \oint_{C}\frac{dz_n }{2\pi i z_n}f(z)g(z^{-1})\Delta(z;q,t),
\end{equation*}
where $C$ is positively oriented and chosen to include $0$ and no other poles. Throughout this paper, it will also be possible to view $f$ and $g$ as formal series, and then the integrals can be defined purely formally as picking out the constant term in this series.

We now give a well-known contour integral formula for the skew Hall--Littlewood polynomials.
\begin{lemma}
\label{lem: HL int formula}
Let $x=(x_1,\dotsc)$ be an alphabet of arbitrary size. If $l(\lambda)\leq n$, then
\begin{equation*}
    P_{\lambda/\mu}(x;0,q)=\frac{(q;q)_{n-l(\lambda)}}{(1-q)^n}\langle P_\lambda(z;0,q),Q_\mu(z;0,q)\Pi(z,x;0,q)\rangle_n'.
\end{equation*}
This equality can either be viewed formally, or if $x_i\in \C$, then the contours defining $\langle,\rangle_n'$ should be chosen to be circles centered at $0$ and include all $x_i$.
\end{lemma}
\begin{proof}
We write
\begin{equation*}
\begin{split}
    P_{\lambda/\mu}(x;0,q)&=\sum_{\nu}\langle P_{\lambda/\mu}(z;0,q),Q_\nu(z;0,q)\rangle P_\nu(x;0,q)
    \\&=\sum_{\nu}\langle P_{\lambda}(z;0,q),Q_\mu(z;0,q)Q_\nu(z;0,q)\rangle P_\nu(x;0,q)
    \\&=\sum_{\nu}\frac{(q;q)_{n-l(\lambda)}}{(1-q)^n}\langle P_{\lambda}(z;0,q),Q_\mu(z;0,q)Q_\nu(z;0,q)\rangle_n' P_\nu(x;0,q),
\end{split}
\end{equation*}
where $\langle,\rangle$ is the usual Macdonald inner product. Since $l(\lambda)\leq n$, the $\langle ,\rangle_n'$ inner product agrees up to the constant with the Macdonald inner product. We then pull the sum inside the inner product and use the Cauchy identity, giving the desired equality. In particular, if the $x_i\in\C$, the choice of contours ensures that the series defining $\Pi$ converges.
\end{proof}

\subsection{Periodic Macdonald processes}
\label{sec: mac process}
We now define the periodic Macdonald process, a measure on sequences of partitions generalizing the periodic $q$-Whittaker measure and periodic Hall--Littlewood process defined in the introduction. We will ultimately only consider the $q=0$ and $t=0$ cases in this paper, but prefer to give a uniform definition. These processes were first considered in \cite{K21}, although the $q=t$ case, known as the periodic Schur measure, was introduced in \cite{B07}.

Let $q,t,u\in [0,1)$, and let $a=(a_1,\dotsc, a_M)$ and $b=(b_1,\dotsc, b_N)$ be parameters in $(0,1)$. We define the \emph{periodic Macdonald measure} on sequences of partitions $\vec\lambda=(\lambda^{(0)}\subseteq\dotsm\subseteq \lambda^{(M)})$ and $\vec\mu=(\mu^{(0)}\subseteq \dotsm\subseteq\mu^{(N)})$ with $\mu^{(0)}=\lambda^{(0)}$ and $\mu^{(N)}=\lambda^{(M)}$ by
\begin{equation}
\label{eq:periodic-macdonald}
    \mathbb{PM}_{q,t,u}^{a,b}(\vec\lambda,\vec\mu)=\frac{u^{|\mu^{(0)}|}}{\Phi(a,b;q,t,u)}\prod_{i=1}^MP_{\lambda^{(i)}/\lambda^{(i-1)}}(a_i;q,t)\prod_{j=1}^N Q_{\mu^{(j)}/\mu^{(j-1)}}(b_j;q,t),
\end{equation}
where
\begin{equation*}
    \Phi(a,b;q,t,u)=\frac{1}{(u;u)_\infty}\prod_{i,j}\frac{(ta_ib_j;q,u)_\infty}{(a_ib_j;q,u)_\infty}.
\end{equation*}
This is actually a special case of the most general definition, which allows for arbitrary positive specializations, and for an arbitrary number of increasing and decreasing sequences of partitions. We will not need this more general notion, and so we give the most concrete version and refer the reader to \cite{K21} for details on the general case.

We will only be interested in two special cases. When $t=0$, we call this the \emph{periodic $q$-Whittaker process}, denoted $\mathbb{PW}_{q,u}^{a,b}$, and when $q=0$, we call this the \emph{periodic Hall--Littlewood process}, denoted $\mathbb{PHL}_{t,u}^{a,b}$. We will refer to the marginal measure of $\lambda^{(M)}$ as the \emph{periodic Macdonald measure}, and similarly for the other cases, and denote these with an extra $\mathbb{M}$, e.g. $\mathbb{PMM}_{q,t,u}^{a,b}$. Finally, when $q=t$, this becomes the \emph{periodic Schur process/measure}, which we denote by $\mathbb{PS}_{u}^{a,b}$/$\mathbb{PSM}_u^{a,b}$. Recall that $\chi$ is $q$-geometric if $\PP(\chi=n)=q^n\frac{(q;q)_\infty}{(q;q)_n}$.

\section{\texorpdfstring{$(q,u)$}{(q,u)} symmetry of the periodic \texorpdfstring{$q$}{q}-Whittaker measure}
\label{sec: sym}
In this section we study the periodic $q$-Whittaker measure via a symmetric function approach, and show Theorem \ref{thm: qt sym}, that after a random shift, the distribution of the first part is symmetric in $q$ and $u$. We begin by stating a contour integral formula, which clearly implies Theorem \ref{thm: qt sym}. We then establish a useful complementation property of $q$-Whittaker functions, and use it to write the distribution function in terms of a contour integral which is clearly symmetric in $q$ and $u$.

\begin{theorem}
\label{thm: qt sym fn}
Let $a=(a_1,\dotsc, a_M)$ and $b=(b_1,\dotsc, b_N)$, let $\lambda\sim \mathbb{PWM}^{a,b}_{q,u}$, and let $\chi$ be an independent $q$-geometric random variable. Let $C$ be a positively oriented circle centered at $0$ and containing all points $a_i$ and $b_j^{-1}$. Then
\begin{equation*}
    \PP(\lambda_1+\chi\leq n)=\frac{\prod_{j=1}^Nb_j^n}{\widetilde{\Phi}(a,b;q,u)}\oint_{C}\frac{dz_1}{2\pi i z_1}\dotsm \oint_C\frac{dz_n}{2\pi i z_n}\prod_{i=1}^n z_i^N\prod_{i,j}(1+z_i^{-1}a_j)\prod_{i,j}(1+z_i^{-1}b_j^{-1})\widetilde{\Delta}(z;q,u),
\end{equation*}
where
\begin{equation*}
    \widetilde{\Phi}(a,b;q,u)=\frac{n!(1-q)^n(1-u)^n}{(q;q)_\infty(u;u)_\infty (1-qu)^n\prod_{i,j}(a_ib_j;q,u)_\infty},
\end{equation*}
and
\begin{equation*}
    \widetilde{\Delta}(z;q,u)=\prod_{i\neq j}\frac{(1-quz_iz_j^{-1})(1-z_iz_j^{-1})}{(1-qz_iz_j^{-1})(1-uz_iz_j^{-1})}.
\end{equation*}
\end{theorem}
Theorem \ref{thm: qt sym} immediately follows from Theorem \ref{thm: qt sym fn}, since the formula given is obviously symmetric in $q$ and $u$. By applying the Macdonald involution, we are also able to obtain formulas for the periodic Hall--Littlewood measure.
\begin{corollary}
\label{cor: HL meas formula}
Let $a=(a_1,\dotsc)$ and $b=(b_1,\dotsc)$, and let $\lambda\sim \mathbb{PHL}_{t,u}^{a,b}$, and let $\chi$ be an independent $t$-geometric random variable. Let $C$ be a positively oriented circle centered at $0$ and containing all points $a_i$ and $b_j^{-1}$ (or interpret the contour integrals formally). Then
\begin{multline*}
    \PP(l(\lambda)+\chi\leq n)=
    \\
    \frac{\Phi(a,b;t,0,u)}{\widetilde{\Phi}(a,b;t,u)\Phi(a,b;0,t,u)}\oint_{C}\frac{dz_1}{2\pi i z_1}\dotsm \oint_C\frac{dz_n}{2\pi i z_n} \prod_{i,j}\frac{1-tz_i^{-1}a_j}{1-z_i^{-1}a_j} \prod_{i,j}\frac{1-tz_ib_j}{1-z_ib_j}\widetilde{\Delta}(z;t,u),
\end{multline*}
\end{corollary}
\begin{proof}
After noticing that $\prod_{i=1}^{n}z_i^N\prod_jb_j^n\prod_{i,j}(1+z_i^{-1}b_j^{-1})=\prod_{i,j}(1+z_ib_j)$, we see that the expression in Theorem \ref{thm: qt sym fn} makes sense for infinite alphabets. We may then apply the Macdonald involution $\omega_{q,0}$ in both the $a$ and $b$ variables, which turns the $q$-Whittaker functions into Hall--Littlewood polynomials, and transposes all partitions. This transforms the dual Cauchy kernels by 
\begin{equation*}
    \omega_{q,0}^{(a)}\omega_{q,0}^{(b)}\prod_{i,j}(1+z_i^{-1}a_j)\prod_{i,j}(1+z_ib_j)=\prod_{i,j}\frac{1-qz_i^{-1}a_j}{1-z_i^{-1}a_j} \prod_{i,j}\frac{1-qz_ib_j}{1-z_ib_j}.
\end{equation*}
After adjusting for the difference in the partition functions for the periodic $q$-Whittaker and Hall--Littlewood measures and replacing $q$ with $t$, we obtain the desired formula.
\end{proof}

\begin{remark}
There is also an interesting identity involving the $\lambda_1$ observable in the periodic Hall--Littlewood measure. Since the proof uses vertex models, we defer it until Section \ref{sec: 6vm} when the relevant models are introduced; see Proposition \ref{prop: PHL Macdonald evaluation}.
\end{remark}

\subsection{Complementation}
A key tool will be the following complementation property of skew $q$-Whittaker polynomials. If $\mu$ is a partition with $\mu_1\leq n$, we write $(n^N,\mu)$ to mean $\mu$ with $N$ additional parts of size $n$.
\begin{proposition}
\label{prop: complement}
Let $N\geq 1$, $x=(x_1,\dotsc, x_N)$, and let $\mu\subseteq \lambda$ such that $\lambda_1\leq n$. Then
\begin{equation*}
    \frac{(q;q)_{n-\mu_1}}{(q;q)_{n-\lambda_1}}Q_{\lambda/\mu}(x;q,0)=P_{(n^N,\mu)/\lambda}(x^{-1};q,0)\prod_{i=1}^Nx_i^n.
\end{equation*}
\end{proposition}

\begin{proof}
The $q$-Whittaker polynomials in a finite alphabet $(x_1,\dots,x_N)$ may be recovered from Hall--Littlewood functions (in an infinite alphabet $x$) under the following specialization:
\begin{align*}
Q_{\lambda/\mu}(x_1,\dots,x_N;q,0)
=
P_{\lambda'/\mu'}(x;0,q)
\Big|_{p_k(x) \mapsto \frac{(-1)^{k+1}}{1-q^k} p_k(x_1,\dots,x_N)}
=:
\omega^{(N)} P_{\lambda'/\mu'}(x;0,q).
\end{align*}
We then have
\begin{align*}
Q_{\lambda/\mu}(x_1,\dots,x_N;q,0)
\frac{(q;q)_{n-\mu_1}}{(q;q)_{n-\lambda_1}}
=
\frac{(q;q)_{n-\ell(\mu')}}{(q;q)_{n-\ell(\lambda')}}
\cdot
\omega^{(N)} 
P_{\lambda'/\mu'}(x;0,q).
\end{align*}
By Lemma \ref{lem: HL int formula}, the right hand side of this expression admits the integral formula
\begin{multline*}
\omega^{(N)} 
\frac{(q;q)_{n-\ell(\mu')}}{(q;q)_{n-\ell(\lambda')}}
\cdot
P_{\lambda'/\mu'}(x;0,q)
=
\\
\omega^{(N)} 
\frac{(q;q)_{n-\ell(\mu')}}{n!(1-q)^n}
\cdot
\oint_{C} \frac{dz_1}{2\pi i z_1}
\cdots
\oint_{C} \frac{dz_n}{2\pi i z_n}
\Delta(z;0,q)
P_{\lambda'}(z;0,q)
Q_{\mu'}(z^{-1};0,q)
\prod_{i=1}^{n}
\prod_{j}
\frac{z_i-qx_j}{z_i-x_j},
\end{multline*}
where the contours are circles centred on the origin which enclose the points $x$. Now compute the action of $\omega^{(N)}$ on the Cauchy kernel, yielding
\begin{multline*}
\omega^{(N)} 
\frac{(q;q)_{n-\ell(\mu')}}{(q;q)_{n-\ell(\lambda')}}
\cdot
P_{\lambda'/\mu'}(x;0,q)
=
\\
\frac{(q;q)_{n-\ell(\mu')}}{n!(1-q)^n}
\cdot
\oint_{C} \frac{dz_1}{2\pi i z_1}
\cdots
\oint_{C} \frac{dz_n}{2\pi i z_n}
\Delta(z;0,q)
P_{\lambda'}(z;0,q)
Q_{\mu'}(z^{-1};0,q)
\prod_{i=1}^{n}
\prod_{j=1}^{N}
(1+x_j/z_i).
\end{multline*}
Extracting the factor $\prod_{j=1}^{N} x_j^n$, and defining new integration variables $u_i = 1/z_i$ for all $1 \leq i \leq n$, this becomes
\begin{multline*}
\omega^{(N)} 
\frac{(q;q)_{n-\ell(\mu')}}{(q;q)_{n-\ell(\lambda')}}
\cdot
P_{\lambda'/\mu'}(x;0,q)
=
\\
\prod_{j=1}^{N} x_j^n
\frac{(q;q)_{n-\ell(\mu')}}{n!(1-q)^n}
\cdot
\oint_{C} \frac{du_1}{2\pi i u_1}
\cdots
\oint_{C} \frac{du_n}{2\pi i u_n}
\Delta(u;0,q)
\\
\prod_{i=1}^{n} u_i^N
P_{\lambda'}(u^{-1};0,q)
Q_{\mu'}(u;0,q)
\prod_{i=1}^{n}
\prod_{j=1}^{N}
(1+x_j^{-1}/u_i).
\end{multline*}
It remains to redistribute overall $q$-dependent factors, and absorb $\prod_{i=1}^{n} u_i^N$ within the second Hall--Littlewood polynomial in the integrand:
\begin{multline*}
\omega^{(N)} 
\frac{(q;q)_{n-\ell(\mu')}}{(q;q)_{n-\ell(\lambda')}}
\cdot
P_{\lambda'/\mu'}(x;0,q)
=
\\
\prod_{j=1}^{N} x_j^n
\frac{(q;q)_{n-\ell(\mu')}}{n!(1-q)^n}
\frac{b_{\mu'}(0,q)}{b_{\lambda'}(0,q)}
\cdot
\oint_{C} \frac{du_1}{2\pi i u_1}
\cdots
\oint_{C} \frac{du_n}{2\pi i u_n}
\Delta(u;0,q)
\\
Q_{\lambda'}(u^{-1};0,q)
P_{\mu'+N^n}(u;0,q)
\prod_{i=1}^{n}
\prod_{j=1}^{N}
(1+x_j^{-1}/u_i),
\end{multline*}
and the right hand side may then be rewritten to yield
\begin{align*}
\omega^{(N)} 
\frac{(q;q)_{n-\ell(\mu')}}{(q;q)_{n-\ell(\lambda')}}
\cdot
P_{\lambda'/\mu'}(x;0,q)
&=
\frac{b_{\mu'}(0,q)}{b_{\lambda'}(0,q)}
\frac{(q;q)_{n-\ell(\mu')}}{(q;q)_{n-\ell(\mu'+N^n)}}
\prod_{j=1}^{N} x_j^n
\cdot
\omega^{(N)}_{x^{-1}}
P_{(\mu'+N^n)/\lambda'}(x^{-1};0,q)
\\
&=
\frac{b_{\mu'}(0,q)}{b_{\lambda'}(0,q)}
(q;q)_{n-\ell(\mu')}
\prod_{j=1}^{N} x_j^n
\cdot
Q_{(n^N,\mu)/\lambda}(x^{-1};q,0),
\end{align*}
using the fact that $\ell(\mu'+N^n) = n$ and $(\mu'+N^n)'=(n^N,\mu)$. The proof is completed by noting that
\begin{align*}
\frac{b_{\mu'}(0,q)}{b_{\lambda'}(0,q)}
(q;q)_{n-\ell(\mu')}
=
\frac{b_{\lambda}(q,0)}{b_{\mu}(q,0)} (q;q)_{n-\mu_1}
=
\frac{b_{\lambda}(q,0)}{b_{(n^N,\mu)}(q,0)}.
\end{align*}

\end{proof}

\subsection{\texorpdfstring{$(q,u)$}{(q,u)} symmetry}
We are now ready to prove Theorem \ref{thm: qt sym fn}. We begin by defining
\begin{equation*}
    Z_{n,N}(q,u;x)=\sum_{\mu: \mu_1 \leq n}\frac{1}{(q;q)_{n-\mu_1}}u^{|\mu|}P_{(n^N,\mu)/\mu}(x;q,0) .
\end{equation*}

\begin{lemma}
\label{lem: prob after complement}
Let $a=(a_1,\dotsc, a_M)$ and $b=(b_1,\dotsc, b_N)$. Let $\lambda\sim \mathbb{PW}_{q,u}^{a,b}$, and let $\chi$ be an independent $q$-geometric random variable. Then we have
\begin{equation*}
    \PP(\lambda_1+\chi\leq n)=\frac{(q;q)_\infty\prod_{j=1}^Nb_j^{n}}{\Phi(a,b;q,0,u)}Z_{n,N}(q,0,u;(a,b^{-1})),
\end{equation*}
where $(a,b^{-1})$ indicates that the two alphabets are concatenated.
\end{lemma}
\begin{proof}
We begin by computing
\begin{equation*}
    \frac{1}{(q;q)_\infty}\Phi(a,b;q,0,u)\PP(\lambda_1+\chi\leq n)=\sum_{\lambda,\mu:\lambda_1\leq n}\frac{1}{(q;q)_{n-\lambda_1}}u^{|\mu|}P_{\lambda/\mu}(a;q,0)Q_{\lambda/\mu}(b;q,0),
\end{equation*}
using the fact that $\PP(\chi\leq k)=\frac{(q;q)_\infty}{(q;q)_{k}}$. If $N=0$, then the statement is clear, as $Q_{\lambda/\mu}(b;q,0)=0$ unless $\lambda=\mu$, immediately giving the expression $Z_{n,N}(q,u;a)$ since $Q_{\mu/\mu}(b;q,0)=1$. We thus assume that $N\geq 1$.

We now use Proposition \ref{prop: complement} on $Q_{\lambda/\mu}$, giving the expression
\begin{equation*}
    \frac{1}{(q;q)_\infty}\Phi(a,b;q,0,u)\PP(\lambda_1+\chi\leq n)=\prod_{j=1}^Nb_j^n\sum_{\lambda,\mu:\lambda_1\leq n}\frac{1}{(q;q)_{n-\mu_1}}u^{|\mu|}P_{\lambda/\mu}(a;q,0)P_{(n^N,\mu)/\lambda}(b^{-1};q,0). 
\end{equation*}
Finally, we may use the branching rule to compute the sum over $\lambda$, which gives the desired equality.
\end{proof}
Notice that after inverting the $b$ variables and clearing denominators, $\PP(\lambda_1+\chi\leq n)$ is actually completely symmetric in the combined alphabet $(a,b)$. We now give a formula for $Z_{n,N}(q,u;x)$.

\begin{lemma}
\label{lem: Z formula}
We have
\begin{equation*}
    Z_{n,N}(q,u;x)=\frac{(1-qu)^n}{n!(1-q)^n(1-u)^n}\oint_{C}\frac{dz_1}{2\pi i z_1}\dotsm \oint_C\frac{dz_n}{2\pi i z_n}\prod_{i=1}^n z_i^N\prod_{i,j}(1+z_i^{-1}x_j)\widetilde{\Delta}(z;q,u),
\end{equation*}
where
\begin{equation*}
    \widetilde{\Delta}(z;q,u)=\prod_{i\neq j}\frac{(1-quz_iz_j^{-1})(1-z_iz_j^{-1})}{(1-qz_iz_j^{-1})(1-uz_iz_j^{-1})}.
\end{equation*}
This can either be interpreted as a formal statement or as contour integrals, with $C$ a positively oriented circle centered at $0$ chosen to include all $x_i$.
\end{lemma}

\begin{proof}
Again, there is a slight difference in the proof if $N=0$, and so for now we assume that $N\geq 1$. Then
\begin{equation}
\label{eq: Z second formula}
    Z_{n,N}(q,u;x) = \sum_{\mu: \mu_1 \leq n} Q_{(n^N,\mu)/\mu}(x;q,0) u^{|\mu|},
\end{equation}
as
\begin{equation}
\label{eq: P to Q}
    \frac{1}{(q;q)_{n-\mu_1}}P_{(n^N,\mu)/\mu}(x;q,0)=Q_{(n^N,\mu)/\mu}(x;q,0).
\end{equation}
We then write
\begin{equation*}
    Z_{n,N}(q,u;x)=\omega_{0,q}  \sum_{\lambda: l(\lambda) \leq n} P_{\lambda+N^n/\lambda}(x;0,q) u^{|\lambda|},
\end{equation*}
where the summand now features a Hall--Littlewood polynomial and $\omega_{0,q}$ is the Macdonald involution. We now use the integral formula given by Lemma \ref{lem: HL int formula}, the fact that
\begin{equation*}
    P_{\lambda+N^n}(z;0,q)=\prod_{i=1}^n z_i^N P_\lambda(z;0,q),
\end{equation*}
and the Cauchy identity. As $(q;q)_{n-l(\lambda+N^n)}=1$, this gives
\begin{equation*}
\begin{split}
    &Z_{n,N}(q,u;x)
    \\=&\omega_{0,q}\sum_{\lambda}\frac{1}{n!(1-q)^n}\oint_{C} \frac{dz_1}{2\pi i z_1}\dotsm \oint_{C}\frac{dz_n}{2\pi i z_n}P_{\lambda+N^n}(z;0,q)Q_\lambda(uz^{-1};0,q)\Pi(z^{-1},x;0,q)\Delta(z;0,q)
    \\=&\omega_{0,q}\sum_{\lambda}\frac{1}{n!(1-q)^n}\oint_{C} \frac{dz_1}{2\pi i z_1}\dotsm \oint_{C}\frac{dz_n}{2\pi i z_n}\prod_i z_i^NP_\lambda(z;0,q)Q_\lambda(uz^{-1};0,q)\Pi(z^{-1},x;0,q)\Delta(z;0,q)
    \\=&\omega_{0,q}\frac{1}{n!(1-q)^n}\oint_{C} \frac{dz_1}{2\pi i z_1}\dotsm \oint_{C}\frac{dz_n}{2\pi i z_n}\prod_i z_i^N\Pi(z,uz^{-1};0,q)\Pi(z^{-1},x;0,q)\Delta(z;0,q),
\end{split}
\end{equation*}
where the restriction $l(\lambda)\leq n$ can be dropped as $Q_\lambda(uz^{-1};0,q)=0$ otherwise. At this point, we recognize that $\omega_{0,q}\Pi(z^{-1},x;0,q)=\prod_{i,j}(1+z_i^{-1}x_j)$ and
\begin{equation*}
    \frac{1}{(1-q)^n}\Pi(z,uz^{-1};0,q)\Delta(z;0,q)=\frac{1}{(1-q)^n}\prod_{i,j}\frac{1-quz_iz_j^{-1}}{1-uz_iz_j^{-1}}\prod_{i\neq j}\frac{1-z_iz_j^{-1}}{1-qz_iz_j^{-1}}.
\end{equation*}
Splitting the $i=j$ factors in the first product gives $\frac{(1-qu)^n}{(1-u)^n}$ and leads to the desired formula.

We now consider the $N=0$ case. Here, \eqref{eq: Z second formula} and \eqref{eq: P to Q} no longer hold, and instead, $P_{\mu/\mu}(x;q,0)=Q_{\mu/\mu}(x;q,0)=1$. Thus, the factor of $\frac{1}{(q;q)_{n-\mu_1}}$ stays. The rest of the proof then proceeds as before, except that when applying Lemma \ref{lem: HL int formula}, a factor of $(q;q)_{n-l(\lambda)}$ is gained where it was not before. Since $\lambda'=\mu$, this exactly cancels the previous factor, and so the contour integral formula remains valid.
\end{proof}

\begin{remark}
A series related to $Z_{n,N}(q,u;x)$ where the $q$-Whittaker polynomials are replaced with Hall--Littlewood polynomials evaluates to a Macdonald polynomial of shape $n^N$ (see Theorem \ref{thm: macdonald vtx model} and Proposition \ref{prop: PHL Macdonald evaluation}). It seems natural to ask whether the $Z_{n,N}$ belong to a family of symmetric functions.
\end{remark}

\begin{proof}[Proof of Theorem \ref{thm: qt sym fn}]
This follows immediately from Lemmas \ref{lem: prob after complement} and \ref{lem: Z formula}.
\end{proof}


\section{Vertex models for the \texorpdfstring{$q$}{q}-Whittaker--periodic Schur correspondence}
\label{sec: t=0}
In this section we give a vertex model interpretation of the $(q,u)$ symmetry of Theorem \ref{thm: qt sym}, in the special case $u=0$. Of course in this case, we recover the known identity of Imamura, Mucciconi, and Sasamoto \cite{IMS21}. While there are now many proofs of this fact, we hope that a vertex model interpretation will potentially lead to further extensions and insights.

Let us begin with a restatement of Theorem \ref{thm: qt sym} in the $u=0$ setting:

\begin{theorem}
\label{thm: evaluation t=0}
Let $a=(a_1,\dotsc, a_M)$ and $b=(b_1,\dotsc, b_N)$. Then if $\lambda\sim \mathbb{PSM}_q^{a,b}$, $\mu\sim\mathbb{PWM}_{q,0}^{a,b}$, and $\chi$ is an independent $q$-geometric random variable, we have
\begin{equation*}
    \PP(\lambda_1\leq n)=\PP(\mu_1+\chi\leq n),
\end{equation*}
or equivalently, as an identity of symmetric functions,
\begin{equation}
\label{eq:IMS}
    \sum_{\lambda,\mu:\lambda_1\leq n}q^{|\mu|}s_{\lambda/\mu}(a)s_{\lambda/\mu}(b)=\sum_{\mu,k:\mu_1+k\leq n}\frac{q^{k}}{(q;q)_k}P_{\mu}(a;q,0)Q_\mu(b;q,0).
\end{equation}
\end{theorem}
Our proof will involve introducing another family of functions $W_\lambda$, originally studied in \cite{GW20}, for which we give a vertex model construction. These functions satisfy a certain symmetry, namely equation \eqref{eq:W-sym}, and this is the key property that ultimately allows the two sides of \eqref{eq:IMS} to be identified.

\subsection{The polynomials \texorpdfstring{$W_{\lambda}$}{Wlambda}}
Following \cite{GW20} we introduce another family of functions, symmetric in two separate alphabets $x$ and $y$. We define $W_\lambda(x;q,t;y)$ as the image of $J_{\lambda}(x;q,t)$ (the latter being the Macdonald polynomials in their integral form; see \cite[Section VI.8]{M79}) under the homomorphism $p_k(x)\mapsto \frac{p_k(x)-(-1)^kp_k(y)}{1-t^k}$. These polynomials satisfy the following symmetry (see Proposition 2 of \cite{GW20}) which is crucial to our approach:
\begin{equation}
\label{eq:W-sym}
W_{\lambda}(x;q,t;y)=W_{\lambda'}(y;t,q;x),  
\end{equation} 
and when $y=-tx$, one has that $W_{\lambda}(x;q,t;-tx)= J_\lambda(x;q,t)$ (this is immediate from the definition $W_{\lambda}(x;q,t;y)$; see also Proposition 3 of \cite{GW20}). We now describe a vertex model construction of these functions.\footnote{A lattice construction of $W_\lambda(x;q,t;y)$ was also obtained in \cite{GW20} in terms of a certain colored bosonic vertex model. Our approach here is slightly different, making use of the fermionic vertex models of \cite{ABW21}, so we spend some time to elaborate on the details.}

\subsubsection{Vertex weights}

Our vertex model will consist of colored arrows traversing a square grid; each arrow will have a color assigned to it, which is an integer label in the set $\{1,\dots,M\}$. Each vertex in the grid consists of the crossing of an oriented horizontal and vertical line, and the four edges connected to the point of intersection are denoted as bottom-incoming, left-incoming, top-outgoing and right-outgoing edges, respectively. Any edge of the lattice may support any number of arrows, but no color may be present more than once at any given edge. 

Let $A,B,C,D\in \{0,1\}^{M}$ be vectors representing the presence or absence of arrows of each color at the edges to the bottom, left, top, and right of a vertex, and fix two indeterminates $x,y$. The weight of a vertex is then defined as
\begin{equation}
\label{eq:weights}
\begin{split}
\begin{tikzpicture}[scale=0.5,baseline=(current bounding box.center),>=stealth]
\draw[lgray,line width=10pt] (-1,0) -- (1,0);
\draw[lgray,line width=10pt] (0,-1) -- (0,1);
\node[left] at (-1,0) {$B$};
\node[right] at (1,0) {$D$};
\node[below] at (0,-1) {$A$};
\node[above] at (0,1) {$C$};
\node[left] at (-1.8,0) {$(x,y) \rightarrow $};
\end{tikzpicture}
&= L_{x,y}(A,B;C,D)
=
\mathbf{1}_{A+B=C+D}
    x^{|D|-|V|}y^{|V|}t^{|V|(1-|D|)+{|B|\choose 2}+\phi(V,D-B)+\phi(D,C)}
    \\
    &
    \qquad
    \times
    \dfrac{\prod_{i=1}^{|D|-|V|}\left(\frac{y}{x}+t^{|V|+i-1}\right)}
    {\prod_{i=1}^{|B|-|V|}\left(\frac{y}{x}+t^{|V|+i-1}\right)}
    \prod_{i:B_i-D_i=1}\left(1+t^{-B_{(i,n]}-D_{[1,i)}}\frac{y}{x}\right)
\end{split}
\end{equation}
where $V=(V_1,\dotsc, V_M)$ is a vector with components given by $V_i=\min(A_i,B_i,C_i,D_i)$ for all $1\leq i \leq M$, $\phi(X,Y) = \sum_{i<j} X_i Y_j$ for any two vectors $X,Y$ and $X_S = \sum_{i \in S} X_i$ for any set $S$. The condition that the weights vanish unless $A+B=C+D$ enforces that colored arrows are conserved as they pass through a vertex. At all times in this section we identify the vertex shown on the left hand side of \eqref{eq:weights} with the function given on the right, and pass between the two quantities in this identification freely. 

\begin{remark}
Although it is not immediately clear from the above definition, the weights \eqref{eq:weights} are in fact polynomials in $x,y,t$ for all $A,B,C,D \in \{0,1\}^M$, although this fact will have no bearing on any aspect of our construction.

The weights \eqref{eq:weights} also obey a Yang--Baxter equation, which is a consequence of the fact that they are obtained via fusion of the weights given in Figure 8.2 of \cite{ABW21}, but we shall also not require this in what follows.
\end{remark}

\subsubsection{Lattice construction}

Fix two alphabets $x=(x_1,\dotsc, x_N)$ and $y=(y_1,\dotsc, y_N)$. We also fix a partition $\lambda$, and let $\lambda^-$ denote the composition obtained by reversing the parts of $\lambda$. Consider the vertex model \eqref{eq:weights} on a semi-infinite lattice, where the number of colors is identified with the length of $\lambda$; namely, $M=l(\lambda)$. The lattice has infinitely many columns, labelled from left to right by the elements of $\mathbb{Z}_{\geq 0}$, but only $N$ rows, with a pair of parameters $(x_i,y_i)$ assigned to the $i$th row (counted from bottom to top). 

Let us now specify the boundary conditions assigned to this lattice. No colored arrows enter the lattice along any of the left-incoming edges, nor leave it via its right-outgoing edges. All of the colors $\{1,\dots,M\}$ enter the lattice via the bottom-incoming edge of the $0$th column; the top-outgoing edge of this column is empty. For each $i \in \{1,\dots,M\}$, an arrow of color $i$ exits the lattice via the top-outgoing edge of column $\lambda^-_i$, and never reappears in any column to the right of this one. In each of the columns $1 \leq j < \lambda^-_i$ the color $i$ may perform a {\it winding}; this means that it both exits via the top-outgoing edge of this column, and re-enters it via its bottom-incoming edge\footnote{Due to the fact that each edge can only be used by a color once, color $i$ can only wind once in any given column, and never in the $0$th column or in column $\lambda^-_i$.}. We associate a winding fugacity of $q^{\lambda^-_i-j}$ to the event that color $i$ winds in column $1 \leq j < \lambda^-_i$. We define the partition function $Z_{\lambda}(x;q,t;y)$ to be the sum over all weighted configurations of arrows satisfying the prescribed boundary conditions, where the weight of a configuration is the product of weights \eqref{eq:weights} that occur at each vertex within the grid and the winding fugacities described above.

Translating the above setup into pictorial form, one has that
\begin{equation}
\label{eq:Z-picture}
Z_{\lambda}(x;q,t;y)
=
\begin{tikzpicture}[scale=0.8,baseline=(current bounding box.center),>=stealth]
\foreach\x in {0,...,6}{
\draw[lgray,line width=10pt] (\x,2) -- (\x,7);
}
\foreach\y in {3,...,6}{
\draw[lgray,line width=10pt] (-1,\y) -- (7,\y);
}
\draw[line width=6pt, ->] (0,1.75) -- (0,2.5);
\node[left] at (-1,6) {\tiny $\emptyset$};
\node[left] at (-1,4.5) {$\vdots$};
\node[left] at (-1,3) {\tiny $\emptyset$};
\node[right] at (7,6) {\tiny $\emptyset$};
\node[right] at (7,4.5) {$\vdots$};
\node[right] at (7,3) {\tiny $\emptyset$};
\node[left] at (-2,6) {$(x_N,y_N) \rightarrow$};
\node[left] at (-2,4.5) {$\vdots$};
\node[left] at (-2,3) {$(x_1,y_1) \rightarrow$};
\node[below] at (0,7.6) {\tiny $\emptyset$};
\node[below] at (1,7.75) {\tiny $\lambda_1^-$};
\node[below] at (3,7.75) {\tiny $\lambda_2^-$};
\node[below] at (4,7.75) {$\cdots$};
\node[below] at (5,7.75) {$\cdots$};
\node[below] at (6,7.75) {\tiny $\lambda_M^-$};
\node[below] at (0,1.5) {\tiny $e_{[1,M]}$};
\end{tikzpicture}
\end{equation}
where $e_{[1,M]} = \sum_{i=1}^M e_i$ with $e_i$ the $i$th Euclidean unit vector in $\mathbb{R}^M$, and where we recall that for each $i \in \{1,\dots,M\}$, the arrow of colour $i$ is allowed to wind in any column $j \in [1,\lambda^-_i)$, such a winding coming with an associated fugacity $q^{\lambda^-_i-j}$.

\begin{proposition}
\label{prop:W-PF}
For all partitions $\lambda$, one has that $Z_{\lambda}(x;q,t;y)=W_{\lambda}(x;q,t;y)$.
\end{proposition}

\begin{proof}

Recall that $W_\lambda(x;q,t;y)$ is defined as the image of $J_{\lambda}(x;q,t)$ under the homomorphism $p_k(x)\mapsto \frac{p_k(x)-(-1)^kp_k(y)}{1-t^k}$. This homomorphism may be concretely realized as follows. Let the alphabet $x$ in $J_{\lambda}(x;q,t)$ be written as the union of $N$ smaller alphabets, of sizes $L_i \geq 1$, $1 \leq i \leq N$:
\begin{equation}
\label{eq:x-union}
x = 
\left(x^{(1)}_1,\dots,x^{(1)}_{L_1}\right) \cup \cdots \cup \left(x^{(N)}_1,\dots,x^{(N)}_{L_N}\right).
\end{equation}
From here perform the substitution $x^{(i)}_j \mapsto x_i t^{j-1}$ for all $1 \leq i \leq N$ and $1 \leq j \leq L_i$. We denote this specialization of the alphabet $x$ by the symbol $\dagger$. Then for any power sum $p_k(x)$ in the original alphabet, one has
\begin{equation*}
p_k(x)^{\dagger}
=
\sum_{i=1}^N x_i^k \frac{1-t^{kL_i}}{1-t^k},
\end{equation*}
and after performing the change of variables $t^{L_i} \mapsto -y_i/x_i$, we read off the relation
\begin{equation*}
p_k(x)^{\dagger} \Big|_{t^{L_i} \mapsto -y_i/x_i}
=
\sum_{i=1}^N \frac{x_i^k-(-1)^k y_i^k}{1-t^k}
=
\frac{p_k(x)-(-1)^kp_k(y)}{1-t^k},
\end{equation*}
which is exactly the homomorphism that maps $J_{\lambda}(x;q,t)$ to $W_\lambda(x;q,t;y)$. It follows that
\begin{equation}
\label{eq:J-spec}
J_{\lambda}(x;q,t)^{\dagger} \Big|_{t^{L_i} \mapsto -y_i/x_i}
=
W_\lambda(x;q,t;y),
\end{equation}
and we now examine the effect of the specializations \eqref{eq:J-spec} on a known vertex model formula for $J_{\lambda}(x;q,t)$, obtained in \cite{ABW21}. 

The partition function in question has very similar boundary conditions to those described above; see \cite[Figure 13.1]{ABW21} for the model being used, and \cite[Equation (15.1.4)]{ABW21} for the partition function itself. For concreteness, we shall denote the partition function of \cite[Equation (15.1.4)]{ABW21} by $\mathfrak{Z}_{\lambda}(x;q,t)$.\footnote{Equation (15.1.4) of \cite{ABW21} uses the partition label $\nu$ instead of $\lambda$, and so $\mathfrak{Z}_{\lambda}(x;q,t)$ shall mean the corresponding object obtained by the notational change $\nu \mapsto \lambda$.} By Theorem 15.1.2 of \cite{ABW21}, one has that $\mathfrak{Z}_{\lambda}(x;q,t)=J_\lambda(x;q,t)$. 

It remains to carry out the specializations \eqref{eq:J-spec} on $\mathfrak{Z}_{\lambda}(x;q,t)$ to obtain a partition function representation of $W_\lambda(x;q,t;y)$. We first note that the weights used in $\mathfrak{Z}_{\lambda}(x;q,t)$ are derived as a specialization of the doubly fused weights given in \cite[Equation (4.3.5)]{ABW21}, with $L=1$, $q^M=s^{-2}$, $y=s$, $q \mapsto t$,\footnote{Throughout most of \cite{ABW21}, $q$ is used as the quantum deformation parameter, which is the reason for this substitution.} dividing each vertex weight by $(-s)^d$ (where $d$ is the number of arrows leaving from the right edge of the vertex), and then taking $s\to 0$. Writing the alphabet $x$ in the form \eqref{eq:x-union} and then principally specializing $x^{(i)}_j \mapsto x_i t^{j-1}$ for all $1 \leq i \leq N$ and $1 \leq j \leq L_i$ instigates {\it fusion} of the horizontal lines of the lattice; the partition function $\mathfrak{Z}_{\lambda}(x;q,t)$ retains its boundary conditions, however, the weights get mapped to the ones described above with $L=L_i$.

In summary, one has that
\begin{equation}
\label{eq:Z-match}
    W_\lambda(x;q,t;y)
    =
    \mathfrak{Z}_{\lambda}(x;q,t)^{\dagger} \Big|_{t^{L_i} \mapsto -y_i/x_i}
\end{equation}
and in view of the above discussion, the right hand side of \eqref{eq:Z-match} is a partition function of the form \eqref{eq:Z-picture} using weights given in \cite[Equation (4.3.5)]{ABW21}, with $L=L_i$, $q^M=s^{-2}$, $y=s$, $q \mapsto t$, $t^{L_i} \mapsto -y_i/x_i$, dividing each vertex weight by $(-s)^d$ (where $d$ is the number of arrows leaving from the right edge of the vertex), and then taking $s\to 0$. These weights are also computed in \cite{ABW21}; they are given by \cite[Proposition 8.2.5]{ABW21} under the substitutions $q \mapsto t$, $r^{-2} \mapsto t^{L_i} = -y_i/x_i$ , $x \mapsto x_i$. Carrying out these substitutions one obtains exactly the weights $L_{x_i,y_i}(A,B;C,D)$ from \eqref{eq:weights}.
\end{proof}

\subsection{The case \texorpdfstring{$\lambda=n^M$}{lambda=nM}}

Since it will have particular significance in what follows, we record here the special case of Proposition \ref{prop:W-PF} where $\lambda$ has rectangular shape. Setting $\lambda=n^M$ for some $n \geq 1$, all of the $M$ colored arrows exit the lattice \eqref{eq:Z-picture} via the top-outgoing edge of column $n$, and  combining Proposition \ref{prop:W-PF} and equation \eqref{eq:Z-picture} we have that
\begin{equation}
\label{eq:rectangle-PF}
W_{n^M}(x;q,t;y)
=
\sum_{P}
\prod_{j=1}^{n-1}
\prod_{i=1}^{M}
q^{jP^{(j)}_i}
\times
\begin{tikzpicture}[scale=0.8,baseline=(current bounding box.center),>=stealth]
\foreach\x in {0,...,6}{
\draw[lgray,line width=10pt] (\x,2) -- (\x,7);
}
\foreach\y in {3,...,6}{
\draw[lgray,line width=10pt] (-1,\y) -- (7,\y);
}
\draw[line width=6pt,->] (0,1.75) -- (0,2.5);
\draw[line width=6pt,->] (6,6.5) -- (6,7.25);
\node[left] at (-1,6) {\tiny $\emptyset$};
\node[left] at (-1,4.5) {$\vdots$};
\node[left] at (-1,3) {\tiny $\emptyset$};
\node[right] at (7,6) {\tiny $\emptyset$};
\node[right] at (7,4.5) {$\vdots$};
\node[right] at (7,3) {\tiny $\emptyset$};
\node[left] at (-1.5,6) {$(x_N,y_N) \rightarrow$};
\node[left] at (-1.5,4.5) {$\vdots$};
\node[left] at (-1.5,3) {$(x_1,y_1) \rightarrow$};
\node[above] at (6,7.25) {\tiny $e_{[1,M]}$};
\node[below] at (5,7.6) {\tiny $P^{(1)}$};
\node[below] at (4,7.6) {\tiny $P^{(2)}$};
\node[below] at (3,7.6) {\tiny $\cdots$};
\node[below] at (2,7.6) {\tiny $\cdots$};
\node[below] at (1,7.6) {\tiny $P^{(n-1)}$};
\node[below] at (0,7.6) {\tiny $\emptyset$};
\node[below] at (0,1.5) {\tiny $e_{[1,M]}$};
\node[below] at (6,2) {\tiny $\emptyset$};
\node[below] at (5,2) {\tiny $P^{(1)}$};
\node[below] at (4,2) {\tiny $P^{(2)}$};
\node[below] at (3,2) {\tiny $\cdots$};
\node[below] at (2,2) {\tiny $\cdots$};
\node[below] at (1,2) {\tiny $P^{(n-1)}$};
\end{tikzpicture}
\end{equation}
using the vertex weights \eqref{eq:weights}, where the sum is taken over all vectors $P^{(j)} \in \{0,1\}^M$, $1 \leq j \leq n-1$ and the factor $q^{jP^{(j)}_i}$ is the usual fugacity associated to the winding of color $i$ in column $n-j$ of the lattice. In this setting, no winding takes place either in the $0$th or $n$th columns of the partition function.

We are now ready to proceed to our vertex model proof of Theorem \ref{thm: evaluation t=0}.

\subsection{Partition function for periodic Schur measure}
In this section, we evaluate the left hand side of equation \eqref{eq:IMS} as a partition function. 

\begin{proposition}
\label{prop: vtx model periodic schur}
Fix two alphabets $a=(a_1,\dots,a_M)$ and $b=(b_1,\dots,b_N)$, and let $(a^{-1},b)$ denote the combined alphabet $(a_1^{-1},\dots,a_M^{-1},b_1,\dots,b_N)$. We have that
\begin{equation}
\label{eq:IMS LHS}
\prod_{i=1}^M a_i^nW_{n^M}(a^{-1},b;q,0;0)=(q;q)_n\sum_{\lambda,\mu:\lambda_1\leq n}q^{|\mu|}s_{\lambda/\mu}(a)s_{\lambda/\mu}(b)
\end{equation}
where the left hand side denotes the polynomial $W_{\lambda}(x;q,t;y)$ with $\lambda=n^M$, $t=0$, $x=(a^{-1},b)$ and $y=0$ (the latter being an alphabet of length $M+N$ in which every parameter is $0$). 
\end{proposition}

We first give the specialization of the vertex model for $W_{n^M}$ which is relevant for this section. Setting $t=0$, $y=0$ in \eqref{eq:weights}, one observes a huge simplification of the model: 
\begin{equation}
\label{eq:weights t=0}
    L_{x,0}(A,B;C,D)\Big|_{t=0}=\mathbf{1}_{A+B=C+D}\mathbf{1}_{V=0}\mathbf{1}_{\phi(D,C+D)=0}x^{|D|},
\end{equation}
where the constraint that $V=0$ arises from the term $y^{|V|}$ and the fact that we set $y=0$, while the requirement that $\phi(D,C+D)=0$ comes from combining all $t$ exponents after setting $y=0$ in \eqref{eq:weights}, which yields $t^{\phi(D,C+D)}$. At the level of the model itself, $\phi(D,C+D)=0$ is equivalent to requiring that at most one arrow exits from the right edge of a vertex, and it must be the largest color to enter the vertex. Therefore, in what follows, we treat the horizontal edges as supporting at most one arrow. We now wish to project away the colors in the model.

To that end, we consider an uncolored vertex model with an arbitrary number of arrows allowed to cross each vertical edge, and at most one arrow allowed to cross each horizontal edge. The weights of this model are defined as follows:
\begin{equation}
\label{eq:uncolored}
\begin{tikzpicture}[scale=0.5,baseline=(current bounding box.center),>=stealth]
\draw[lgray,ultra thick] (-1,0) -- (1,0);
\draw[lgray,line width=10pt] (0,-1) -- (0,1);
\node[left] at (-1,0) {$b$};
\node[right] at (1,0) {$d$};
\node[below] at (0,-1) {$a$};
\node[above] at (0,1) {$c$};
\node[left] at (-1.8,0) {$x \rightarrow $};
\end{tikzpicture}
=
\mathbf{1}_{a+b=c+d}
x^d,
\qquad
a,c \in \mathbb{Z}_{\geq 0},
\qquad
b,d \in \{0,1\},
\end{equation}
where we now use a thin horizontal line in reflection of the fact that these edges support at most one arrow. Using this model, we construct a partition function $\mathfrak{Z}_M(q;t)$ in a similar vein to \eqref{eq:rectangle-PF}: $M$ arrows enter via the bottom-incoming edge of the $0$th column and leave via the top-outgoing edge of the $n$th column. Arrows are allowed to wind an arbitrary number of times except in the $0$th and $n$th columns. The $i$th row of the lattice has a parameter $x_i$ associated to it, so that each vertex within that row has a weight given by \eqref{eq:uncolored} with $x \mapsto x_i$. The weight of a configuration is then the product of the vertex weights that comprise it, and a factor of $q^{n-j}$ each time an arrow winds in column $j$. More precisely, we have that
\begin{equation}
\label{eq:Z-uncolored}
\mathfrak{Z}_M(q;t)
=
\sum_{m}
\prod_{j=1}^{n-1}
q^{j m_j}
\times
\begin{tikzpicture}[scale=0.8,baseline=(current bounding box.center),>=stealth]
\foreach\x in {0,...,6}{
\draw[lgray,line width=10pt] (\x,2) -- (\x,7);
}
\foreach\y in {3,...,6}{
\draw[lgray,ultra thick] (-1,\y) -- (7,\y);
}
\draw[line width=6pt, ->] (0,1.75) -- (0,2.5);
\draw[line width=6pt, ->] (6,6.5) -- (6,7.25);
\node[left] at (-1.5,6) {$x_N \rightarrow $};
\node[left] at (-1,4.5) {$\vdots$};
\node[left] at (-1.5,3) {$x_1 \rightarrow $};
\node[left] at (-1,3) {\tiny $0$};
\node[left] at (-1,6) {\tiny $0$};
\node[right] at (7,3) {\tiny $0$};
\node[right] at (7,4.5) {\tiny $\vdots$};
\node[right] at (7,6) {\tiny $0$};
\node[below] at (0,7.75) {\tiny $0$};
\node[below] at (1,7.75) {\tiny $m_{n-1}$};
\node[below] at (3,7.75) {\tiny $\cdots$};
\node[below] at (4,7.75) {\tiny $m_2$};
\node[below] at (5,7.75) {\tiny $m_1$};
\node[below] at (6,8) {\tiny $M$};
\node[below] at (0,1.5) {\tiny $M$};
\node[below] at (1,2) {\tiny $m_{n-1}$};
\node[below] at (3,2) {\tiny $\cdots$};
\node[below] at (4,2) {\tiny $m_2$};
\node[below] at (5,2) {\tiny $m_1$};
\node[below] at (6,2) {\tiny $0$};
\end{tikzpicture}
\end{equation}
using the vertex weights \eqref{eq:uncolored}, where the sum is taken over all integers $m_j \in \mathbb{Z}_{\geq0}$, $1 \leq j \leq n-1$, and the factor $q^{jm_j}$ is the aforementioned fugacity associated to $m_j$ windings in column $n-j$ of the lattice.

We now show that the two partition functions \eqref{eq:rectangle-PF} and \eqref{eq:Z-uncolored} are related as
\begin{equation}
\label{eq:col-uncol-match}
W_{n^M}(x;q,0;0)
=
(q;q)_{n-1}
\mathfrak{Z}_M(q;t).   
\end{equation} 
This will be done by matching the two sides of the equation column by column, using the following lemma.

\begin{lemma}
\label{lem: periodic schur color blind}
Fix an integer $1 \leq j \leq n-1$ and consider the $j$th column of the partition function \eqref{eq:rectangle-PF} at $t=0$, $y_i=0$ for all $1 \leq i \leq N$. Assume that arrows of color $i_1,\dots,i_M$ enter this column via rows $1 \leq s_1 < \cdots < s_M \leq N$ respectively. Then given a set of outgoing edges $1 \leq r_1 < \cdots < r_M \leq N$ of the column, there exists a unique way to assign the colors $i_1,\dots,i_M$ to these edges such that the weight of the column is non-zero. 

Further, the weight of this column is then equal to the weight of the $j$th column in the uncolored partition function \eqref{eq:Z-uncolored}, in which arrows enter the column via rows $1 \leq s_1 < \cdots < s_M \leq N$ and leave it via rows $1 \leq r_1 < \cdots < r_M \leq N$, up to an overall multiplicative factor of $1-q^{n-j}$.

A similar statement matches the $0$th and $n$th columns of \eqref{eq:rectangle-PF} and \eqref{eq:Z-uncolored}, except that colors enter/exit from the bottom/top respectively rather than the left/right, and no extra factor is needed.
\end{lemma}
\begin{proof}
The uniqueness of the assignment of outgoing colors is because for each vertex in the model \eqref{eq:weights t=0}, only the largest color can exit to the right. Then the largest color must exit the column at the first possible opportunity, and inductively this shows that the assignment of colors to the exit rows is unique. It is also clear that this assignment respects the condition that $V=0$. Moreover, since this is a greedy assignment of exit rows, winding occurs exactly once for each entrance row with no exit row above it (recall that each such winding comes with fugacity $q^{n-j}$). It follows that the weight of the column in this unique configuration is given by 
$q^{(n-j)s}\prod_{i=1}^{M} x_{r_i}$, where $s = |\{k: s_k > r_M\}|$.

Now consider the same column within the uncolored partition function \eqref{eq:Z-uncolored}. Notice that the number of times $s$ that arrows wind in the colored model is the minimal number of windings that occurs in the uncolored model. One finds that in the uncolored model, the weight of the column is the same except that the arrows can wind an arbitrary number of times before exiting. It follows that the weight of the column in the uncolored model is given by $\sum_{i=s}^{\infty}q^{(n-j)i}\prod_{i=1}^{M} x_{r_i}$, and the match of two weights (modulo the factor of $1-q^{n-j}$) is immediate after summing the latter geometric series.

In the $0$th and $n$th columns, the same argument applies, but there is no winding so no extra factor is needed.
\end{proof}

Using Lemma \ref{lem: periodic schur color blind} across columns $0,\dots,n$ of the partition functions \eqref{eq:rectangle-PF} and \eqref{eq:Z-uncolored}, we conclude that \eqref{eq:col-uncol-match} holds.

\begin{proof}[Proof of Proposition \ref{prop: vtx model periodic schur}]
We begin by writing the left hand side of \eqref{eq:IMS LHS} as a partition function of the uncolored model \eqref{eq:uncolored}, as described above. Our main tool is equation \eqref{eq:col-uncol-match} with $x \mapsto (a^{-1},b)$. We first adjust $0$th column in this formula, by allowing the arrows to enter from the bottom $M$ rows on their left-incoming edges, and allowing winding with a factor of $q^n$ for each time an arrow winds. It is easy to see that this only changes the partition function by a factor of $1-q^n$. We then have

\begin{equation}
\label{eq:schur-final-match}
\prod_{i=1}^M a_i^nW_{n^M}(a^{-1},b;q,0;0)=
(q;q)_n\prod_{i=1}^M a_i^n
\sum_{m}
\prod_{j=1}^{n} q^{j m_j}
\times
\begin{tikzpicture}[scale=0.8,baseline=(current bounding box.center),>=stealth]
\foreach\x in {0,...,6}{
\draw[lgray,line width=10pt] (\x,0) -- (\x,7);
}
\foreach\y in {1,...,6}{
\draw[lgray,ultra thick] (-1,\y) -- (7,\y);
}
\draw[ultra thick, ->] (-1,1) -- (0,1);
\draw[ultra thick, ->] (-1,2) -- (0,2);
\draw[ultra thick, ->] (-1,3) -- (0,3);
\draw[line width=6pt, ->] (6,6.5) -- (6,7.25);
\node[left] at (-1,3) {$a_M^{-1}$};
\node[left] at (-1,2) {$\vdots$};
\node[left] at (-1,1) {$a_1^{-1}$};
\node[left] at (-1,6) {$b_N$};
\node[left] at (-1,5) {$\vdots$};
\node[left] at (-1,4) {$b_1$};
\node[below] at (0,7.75) {$m_n$};
\node[below] at (3,7.75) {$\cdots$};
\node[below] at (4,7.75) { $m_2$};
\node[below] at (5,7.75) {$m_1$};
\node[below] at (6,8) {$M$};
\node[below] at (0,0) {$m_{n}$};
\node[below] at (3,0) {$\cdots$};
\node[below] at (4,0) {$m_2$};
\node[below] at (5,0) {$m_1$};
\node[below] at (6,0) {$0$};
\end{tikzpicture}
\end{equation}
where the sum is taken over all integers $m_j \in \mathbb{Z}_{\geq 0}$, $1 \leq j \leq n$.

Let $\lambda$ and $\mu$ be partitions such that $m_k(\lambda)$ denotes the number of vertical arrows crossing the middle of the $(n-k)$th column of \eqref{eq:schur-final-match}, and $m_k(\mu)$ the number of vertical arrows crossing the top of the $(n-k)$th column. We remove the $n$th column from the picture entirely, which is permissible since it has weight $1$ in all configurations. As such, arrows which previously exited the lattice via the top-outgoing edge of this column now exit the lattice freely via the right-outgoing edges of the penultimate column.

Next, we perform a well-known complementation procedure on the horizontal edges of the bottom $M$ rows of the resulting partition function, replacing any horizontal arrow with no arrow and vice versa in any configuration. This has the effect of making arrows move left and up within this portion of the lattice, rather than right and up. Moreover, distributing the $\prod_{i=1}^{M} a_i^n$ factor among the vertices, we see the same change reflected in the vertex weights \eqref{eq:uncolored}, replacing $a_i^{-d}$ with $a_i^{1-d}$ (where $d$ is the state on the right horizontal edge of a vertex).

The result of all of these considerations is the equation

\begin{equation*}
\prod_{i=1}^M a_i^nW_{n^M}(a^{-1},b;q,0;0)=
(q;q)_n \sum_{\lambda,\mu:\lambda_1\leq n}q^{|\mu|} \times
\begin{tikzpicture}[scale=0.8,baseline=(current bounding box.center),>=stealth]

\foreach\x in {0,...,6}{
\draw[lred,line width=10pt] (\x,0) -- (\x,3.5);
}
\foreach\x in {0,...,6}{
\draw[lgray,line width=10pt] (\x,3.5) -- (\x,7);
}
\foreach\y in {4,...,6}{
\draw[lgray,ultra thick] (-1,\y) -- (7,\y);
}
\foreach\y in {1,...,3}{
\draw[lred,ultra thick] (-1,\y) -- (7,\y);
}
\draw[ultra thick, ->] (7,1) -- (6,1);
\draw[ultra thick, ->] (7,3) -- (6,3);
\draw[ultra thick, ->] (6,5) -- (7,5);
\draw[ultra thick, ->] (6,6) -- (7,6);
\node[left] at (-1,3) {$a_M$};
\node[left] at (-1,2) {$\vdots$};
\node[left] at (-1,1) {$a_1$};
\node[left] at (-1,6) {$b_N$};
\node[left] at (-1,5) {$\vdots$};
\node[left] at (-1,4) {$b_1$};
\node[below] at (0,7.75) {\tiny$m_n(\mu)$};
\node[below] at (2,7.75) {$\cdots$};
\node[below] at (6,7.75) {\tiny$m_1(\mu)$};
\node[below] at (5,7.75) {\tiny$m_2(\mu)$};
\node[below] at (0,-0.25) {\tiny$m_n(\mu)$};
\node[below] at (2,-0.25) {$\cdots$};
\node[below] at (6,-0.25) {\tiny$m_1(\mu)$};
\node[below] at (5,-0.25) {\tiny$m_2(\mu)$};
\node[below] at (0,4) {\tiny$m_n(\lambda)$};
\node[below] at (2,4) {$\cdots$};
\node[below] at (6,4) {\tiny$m_1(\lambda)$};
\node[below] at (5,4) {\tiny$m_2(\lambda)$};
\end{tikzpicture}
\end{equation*}
using the fact that $\prod_{j=1}^{n}q^{j m_j} = q^{|\mu|}$. The right edges of this lattice are free, and considered to be summed over all possible ways for arrows to enter/exit. The red vertices used in the the bottom half of the lattice have modified weights that reflect the complementation procedure explained above; namely, they are given by \eqref{eq:uncolored} with $b \mapsto 1-b$ and $d \mapsto 1-d$.

We now see that the bottom half of the partition function evaluates to $s_{\lambda/\mu}(a)$ and the top half evaluates to $s_{\lambda/\mu}(b)$ (the top half of the model is a special case of Lemma \ref{lem: boson HL} with $t=0$, and the bottom is as well, up to reflection in the horizontal Cartesian axis and reversing the direction of all arrows). We recognize the right hand side of \eqref{eq:IMS LHS}.
\end{proof}

\subsection{Partition function for \texorpdfstring{$q$}{q}-Whittaker measure}

In this section, we evaluate the right hand side of \eqref{eq:IMS} as a partition function. Together with Proposition \ref{prop: vtx model periodic schur}, this allows us to complete the proof of Theorem \ref{thm: evaluation t=0}.
\begin{proposition}
\label{prop: vtx model q-whit}
Fix two alphabets $a=(a_1,\dots,a_M)$ and $b=(b_1,\dots,b_N)$, and let $(a^{-1},b)$ denote the combined alphabet $(a_1^{-1},\dots,a_M^{-1},b_1,\dots,b_N)$. We have that
\begin{equation}
\label{eq:W-qW}
    \prod_{i=1}^M a_i^nW_{M^n}(0;0,q;a^{-1},b)=\sum_{\mu:\mu_1\leq n}\frac{(q;q)_n}{(q;q)_{n-\mu_1}}P_\mu(a;q,0)Q_\mu(b;q,0),
\end{equation}
where the left hand side denotes the polynomial $W_{\lambda}(x;q,t;y)$ with $\lambda=M^n$, $q=0$, $t \mapsto q$, $y=(a^{-1},b)$ and $x=0$ (the latter being an alphabet of length $M+N$ in which every parameter is $0$).
\end{proposition}

We first give the specialization of the vertex model for $W_{M^n}$ which is relevant for this section. The first thing to note is that we work with a model consisting of $n$ colors, rather than $M$ as previously, in view of the fact that our partition is now $M^n$. Setting $x=0$, $y \mapsto x$ and $t \mapsto q$ in \eqref{eq:weights} yields, after collecting $q$ exponents,
\begin{equation*}
    L_{0,x}(A,B;C,D) \Big|_{t \mapsto q}
    =\mathbf{1}_{A+B=C+D} \mathbf{1}_{B\leq C}x^{|D|}q^{\phi(D,C-B)}
\end{equation*}
where $B\leq C$ means that $B_i\leq C_i$ for all $1 \leq i \leq n$. Since the winding fugacities are sent to zero in the left hand side of \eqref{eq:W-qW}, no arrows are permitted to perform windings in the partition function representation of $W_{M^n}$. We shall then use the following color blindness property to project away the colors.
\begin{lemma}
\label{lem: q-whit color proj}
Fix four integers $a,b,c,d \in \mathbb{Z}_{\geq 0}$ as well as two vectors $A,B \in \{0,1\}^n$ such that $|A|=a$, $|B|=b$. The following identity holds:
\begin{equation}
\label{eq: q-whit uncolored vtx wt}
    \sum_{C,D:|C|=c,|D|=d}L_{0,x}(A,B;C,D)\Big|_{t \mapsto q}
    =\mathbf{1}_{a+b=c+d} \mathbf{1}_{b\leq c} x^d {a\choose d}_q.
\end{equation}
\end{lemma}
\begin{proof}
We compute
\begin{equation*}
    \sum_{C,D:|C|=c,|D|=d}L_{0,x}(A,B;C,D)\Big|_{t \mapsto q}
    =\mathbf{1}_{a+b=c+d} \mathbf{1}_{b\leq c}x^d\sum_{D\leq A,|D|=d}q^{\phi(D,A-D)}=\mathbf{1}_{a+b=c+d} \mathbf{1}_{b\leq c}x^d{a\choose d}_q.
\end{equation*}
\end{proof}

In the partition function representation of $W_{M^n}(0;0,q;a^{-1},b)$, all colors leave the lattice via the same top-outgoing edge. This means that we can project away the colors of the model, using Lemma \ref{lem: q-whit color proj} vertex by vertex, giving a partition function with uncolored weights defined by \eqref{eq: q-whit uncolored vtx wt}.

\begin{proof}[Proof of Proposition \ref{prop: vtx model q-whit}]
Our main tool will be matching with known vertex model expressions for the $q$-Whittaker polynomials, obtained in \cite{BW21}. We begin by writing the left hand side of \eqref{eq:W-qW} as a partition function in the vertex model with weights given by \eqref{eq: q-whit uncolored vtx wt}. We find that 

\begin{equation*}
\prod_{i=1}^M a_i^nW_{M^n}(0;0,q;a^{-1},b)
=
\prod_{i=1}^M a_i^n
\times
\begin{tikzpicture}[scale=1,baseline=(current bounding box.center),>=stealth]
\foreach\x in {0,...,6}{
\draw[lgray,line width=10pt] (\x,0) -- (\x,7);
}
\foreach\y in {1,...,6}{
\draw[lgray,line width=10pt] (-1,\y) -- (7,\y);
}
\draw[line width=6pt, ->] (0,-0.25) -- (0,0.5);
\draw[line width=6pt, ->] (6,6.5) -- (6,7.25);
\node[left] at (-1,6) {$a_M^{-1}$};
\node[left] at (-1,5) {$\vdots$};
\node[left] at (-1,4) {$a_1^{-1}$};
\node[left] at (-1,3) {$b_N$};
\node[left] at (-1,2) {$\vdots$};
\node[left] at (-1,1) {$b_1$};
\node[below] at (1,7.75) {$0$};
\node[below] at (2,7.75) {$0$};
\node[below] at (3,7.75) {$\cdots$};
\node[below] at (0,7.75) {$0$};
\node[below] at (5,7.75) {$0$};
\node[below] at (6,7.75) {$n$};
\node[below] at (0,-0.5) {$n$};
\node[below] at (1,-0.5) {$0$};
\node[below] at (2,-0.5) {$0$};
\node[below] at (3,-0.5) {$\cdots$};
\node[below] at (5,-0.5) {$0$};
\node[below] at (6,-0.5) {$0$};
\node[below] at (0,3.75) {\footnotesize $n{-}\mu_1$};
\node[below] at (1,3.75) {\footnotesize $\mu_1{-}\mu_2$};
\node[below] at (2,3.75) {\footnotesize $\mu_2{-}\mu_3$};
\node[below] at (3,3.75) {$\cdots$};
\end{tikzpicture}
\end{equation*}
where we use the symmetry in the combined alphabet $(a^{-1},b)$ to place $a^{-1}_i$ dependent rows towards the top and $b_i$ dependent rows towards the bottom, which is the opposite of the convention we adopted previously. The lattice consists of $M+1$ columns, labelled (from left to right) as $0$ through to $M$. We let $\mu$ be a partition such that $n-\mu_1$ counts the number of arrows vertically traversing the middle of the $0$th column, and $\mu_i-\mu_{i+1}$ counts the number of arrows vertically traversing the middle of the $i$th column, as shown in the picture.

Consider firstly the bottom half of this partition function. Reflecting that portion of the lattice in the vertical Cartesian axis and rotating by $180^\circ$, the weights match those of Equation (38) in \cite{BW21} with $s=0$, and the geometry of the lattice matches the right panel of Figure 5 of \cite{BW21}. Moreover, while the partition function in \cite{BW21} has a $0$th column with infinitely many arrows entering via the top-incoming edge, exactly $n$ will actually enter remaining columns of the lattice. Replacing the infinitely many arrows of \cite{BW21} with our $n$ arrows amounts to gaining a $\frac{(q;q)_{a}}{(q;q)_{a-d}}$ factor for each vertex in the $0$th column, which results in an extra overall $\frac{(q;q)_n}{(q;q)_{n-\mu_1}}$ factor. Thus, the bottom half of the lattice gives $\frac{(q;q)_{n}}{(q;q)_{n-\mu_1}}Q_{\mu}(b;q,0)$.

For the top half of the lattice, we note that after conjugating the weights of Equation (35) in \cite{BW21} at $s=0$ by $x^{j-\ell}\frac{(q;q)_\ell}{(q;q)_j}$ and rotating by $180^\circ$, one obtains a match with our weights. Similarly, this transformation maps the geometry of the left panel of Figure 5 of \cite{BW21} to that of the top half our lattice. Again, there is a slight difference due to the $0$th column being used in \cite{BW21} having infinitely many arrows, but here this discrepancy results in no extra overall factors. Thus, the top half of the lattice gives $\prod_{i=1}^{M} a_i^n P_{M^n-\mu}(a^{-1};q,0)$. After applying Proposition \ref{prop: complement}, we thus recover the right hand side of \eqref{eq:W-qW}.
\end{proof}

\begin{remark}
Proposition \ref{prop: vtx model q-whit} could also be proved using the branching rule and complementation given by Proposition \ref{prop: complement}. 
\end{remark}

\begin{proof}[Proof of Theorem \ref{thm: evaluation t=0}]
This follows immediately from Propositions \ref{prop: vtx model periodic schur} and \ref{prop: vtx model q-whit}, and the symmetry $W_{\lambda}(x;q,t;y)=W_{\lambda'}(y;t,q;x)$.
\end{proof}

\section{Periodic Hall--Littlewood process and quasi-periodic six vertex model}
\label{sec: 6vm}
In this section we study the periodic Hall--Littlewood process. We define a quasi-periodic stochastic six vertex model, and show that certain observables of the periodic Hall--Littlewood process match those of the quasi-periodic six vertex model in distribution. This generalizes a result in \cite{BBW16} matching observables of the ordinary Hall--Littlewood process with those of the six vertex model on a rectangular domain. We also show that in the $u\to 1$ limit, the quasi-periodic six vertex model converges to the stationary periodic six vertex model, and explain what parts of the connection to the periodic Hall--Littlewood process survive this procedure.

\subsection{Quasi-periodic six vertex model}
\begin{figure}
\label{fig:6v}
    \centering
    \begin{tabular}{l|cccccc}
    Configuration:&$\vcenter{\hbox{\includegraphics[scale=1]{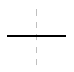}}}$ &$\vcenter{\hbox{\includegraphics[scale=1]{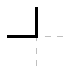}}}$&$\vcenter{\hbox{\includegraphics[scale=1]{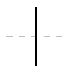}}}$&$\vcenter{\hbox{\includegraphics[scale=1]{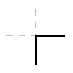}}}$&$\vcenter{\hbox{\includegraphics[scale=1]{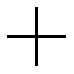}}}$& $\vcenter{\hbox{\includegraphics[scale=1]{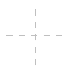}}}$\\
    Probability:&$\mathbf{p}_{i,j}$&$1-\mathbf{p}_{i,j}$&$t\mathbf{p}_{i,j}$&$1-t\mathbf{p}_{i,j}$ &1&1
    \end{tabular}
    \caption{Vertex weights for a vertex $(i,j)$. The black line represents an arrow and the dashed line represents no arrow.}
    \label{fig:vtx wts}
\end{figure}

Let $M,N\in\N$, let $t\in [0,1)$, and let $a=(a_1,\dotsc, a_M)$ and $b=(b_1,\dotsc,b_N)$ be parameters in $(0,1)$, which we call \emph{rapidities}. The \emph{stochastic six vertex model} on an $M\times N$ rectangle is a probability measure on configurations of arrows in an $M\times N$ lattice. Arrows travel right and upwards, and we specify the locations of the incoming arrows along the left and bottom. Each edge in the lattice can be occupied by at most one arrow. We let
\begin{equation*}
    \mathbf{p}_{i,j}=\frac{1-a_ib_j}{1-ta_ib_j},
\end{equation*}
and the probability of any configuration is then a product of vertex weights given in Figure \ref{fig:vtx wts}. We will frequently use graphical notation to denote this model, and in particular will use
\begin{equation*}
\begin{tikzpicture}[scale=0.8,baseline=(current bounding box.center),>=stealth]
\begin{scope}[rotate around={45:(0,0)}]
\draw[thick, dotted] (8.5,4.5)node[left]{\tiny $T_2(N)$} -- (11,2) node[right] {\tiny $S_2(N)$};
\draw[thick, dotted] (8,4) -- (10.5,1.5);
\draw (10.8,1.35)node{$\vdots$};
\draw ({10.8-3},{1.35+3})node{$\vdots$};
\draw[thick, dotted] (7.5,3.5)node[left]{\tiny $T_2(1)$} -- (10,1) node[right] {\tiny $S_2(1)$};
\draw[ultra thick,->] (10.5,2.5) -- (11,2);
\draw[ultra thick,->] (8.5,4.5) -- (8.8,4.2);
\draw[ultra thick,->] (8,4) -- (8.3,3.7);
\draw[ultra thick,->] (7.5,3.5) -- (7.8,3.2);
\draw[thick, dotted] (7.5,2.5)node[below]{\tiny$T_1(1)$} -- (9.5,4.5) node[above] {\tiny $S_1(1)$};
\draw[thick, dotted] (8,2) -- (10,4);
\draw[thick, dotted] (8.5,1.5) -- (10.5,3.5);
\draw (10.25,3.75)node[above]{$\dots$};
\draw (8,1.5)node[]{$\dots$};
\draw[thick, dotted] (9,1)node[below]{\tiny$T_1(M)$} -- (11,3) node[above] {\tiny $S_1(M)$};
\draw[ultra thick,->] (10,3) -- (10.5,3.5); \draw[ultra thick,->] (9,4) -- (9.5,4.5);
\draw[rotate around={-45:(7.75,3)},fill=white] (7.75,3)rectangle({7.75+3.5*sqrt(2)/2},{3+2.5*sqrt(2)/2})node[midway]{$z$};
\end{scope}
\end{tikzpicture}
\end{equation*}
to denote the probability for the six vertex model in an $M\times N$ domain to have outgoing arrows at locations specified by $S_1$ and $S_2$ along the top and right, given incoming arrows $T_1$ and $T_2$ on the bottom and left, respectively. Here we are implicitly labelling the columns and rows with rapidities $za_i$ and $b_j$, where $z$ is a parameter that varies in each copy of our underlying $M \times N$ domain, and will explicitly write this when necessary. We may sometimes also simply write $S_i$ and $T_i$ once as shorthand or omit them entirely if they are not important, but it will always be assumed that outer edges are fully specified.

We will also wish to combine these diagrams into more complicated pictures. The general rule is that internal edges are always summed over. For example,
\begin{equation*}
\begin{tikzpicture}[scale=0.8,baseline=(current bounding box.center),>=stealth]
\begin{scope}[xshift=3.6cm,rotate around={45:(0,0)}]
\draw[thick, dotted] (8.5,4.5) -- (11,2) ;
\draw[thick, dotted] (8,4) -- (10.5,1.5);
\draw[thick, dotted] (7.5,3.5) -- (10,1);
\draw[ultra thick,->] (10.5,2.5) -- (11,2);
\draw[thick, dotted] (7.5,2.5) -- (9.5,4.5) ;
\draw[thick, dotted] (8,2) -- (10,4);
\draw[thick, dotted] (8.5,1.5) -- (10.5,3.5);
\draw[thick, dotted] (9,1) -- (11,3);
\draw[rotate around={-45:(7.75,3)},fill=white] (7.75,3)rectangle({7.75+3.5*sqrt(2)/2},{3+2.5*sqrt(2)/2})node[midway]{$z_2$};
\end{scope}
\begin{scope}[rotate around={45:(0,0)}]
\draw[thick, dotted] (8.5,4.5) -- (11,2) ;
\draw[thick, dotted] (8,4) -- (10.5,1.5);
\draw[thick, dotted] (7.5,3.5) -- (10,1);
\draw[ultra thick,->] (8.5,4.5) -- (8.8,4.2);
\draw[ultra thick,->] (8,4) -- (8.3,3.7);
\draw[ultra thick,->] (7.5,3.5) -- (7.8,3.2);
\draw[thick, dotted] (7.5,2.5) -- (9.5,4.5) ;
\draw[thick, dotted] (8,2) -- (10,4);
\draw[thick, dotted] (8.5,1.5) -- (10.5,3.5);
\draw[thick, dotted] (9,1) -- (11,3);
\draw[ultra thick,->] (10,3) -- (10.5,3.5); \draw[ultra thick,->] (9,4) -- (9.5,4.5);
\draw[rotate around={-45:(7.75,3)},fill=white] (7.75,3)rectangle({7.75+3.5*sqrt(2)/2},{3+2.5*sqrt(2)/2})node[midway]{$z_1$};
\end{scope}
\end{tikzpicture}=\sum_{S}\begin{tikzpicture}[scale=0.8,baseline=(current bounding box.center),>=stealth]
\begin{scope}[xshift=4.2cm,rotate around={45:(0,0)}]
\draw[thick, dotted] (8.5,4.5) -- (11,2) ;
\draw[thick, dotted] (8,4)node[left]{$S$} -- (10.5,1.5);
\draw[thick, dotted] (7.5,3.5) -- (10,1);
\draw[ultra thick,->] (10.5,2.5) -- (11,2);
\draw[thick, dotted] (7.5,2.5) -- (9.5,4.5) ;
\draw[thick, dotted] (8,2) -- (10,4);
\draw[thick, dotted] (8.5,1.5) -- (10.5,3.5);
\draw[thick, dotted] (9,1) -- (11,3);
\draw[rotate around={-45:(7.75,3)},fill=white] (7.75,3)rectangle({7.75+3.5*sqrt(2)/2},{3+2.5*sqrt(2)/2})node[midway]{$z_2$};
\end{scope}
\begin{scope}[rotate around={45:(0,0)}]
\draw[thick, dotted] (8.5,4.5) -- (11,2) ;
\draw[thick, dotted] (8,4) -- (10.5,1.5);
\draw[thick, dotted] (7.5,3.5) -- (10,1);
\draw[ultra thick,->] (8.5,4.5) -- (8.8,4.2);
\draw[ultra thick,->] (8,4) -- (8.3,3.7);
\draw[ultra thick,->] (7.5,3.5) -- (7.8,3.2);
\draw[thick, dotted] (7.5,2.5) -- (9.5,4.5) ;
\draw[thick, dotted] (8,2) -- (10,4);
\draw[thick, dotted] (8.5,1.5) -- (10.5,3.5);
\draw[thick, dotted] (9,1) -- (11,3);
\draw[ultra thick,->] (10,3) -- (10.5,3.5); \draw[ultra thick,->] (9,4) -- (9.5,4.5);
\draw[rotate around={-45:(7.75,3)},fill=white] (7.75,3)rectangle({7.75+3.5*sqrt(2)/2},{3+2.5*sqrt(2)/2})node[midway]{$z_1$};
\end{scope}
\end{tikzpicture}
\end{equation*}
where we sum over all possible ways for arrows to exit the first six vertex model configuration and then enter the second.

Let $u\in [0,1)$ and $L\in\N$, and continue to use the previously defined parameters. We define the \emph{quasi-periodic stochastic six vertex model} of length $L$ to be the probability distribution on $L$ six vertex models on $M\times N$ domains with the $l$th one having rapidities $u^{l-1}a$ and $b$, and such that the arrows exiting on the right of the $l$th one enter the left of the $(l-1)$th one, and the arrows exiting on the top of the $l$th one enter the bottom of the $(l-1)$th one. This will be represented with the picture
\begin{equation*}
\begin{tikzpicture}[scale=0.8,baseline=(current bounding box.center),>=stealth]
\begin{scope}[xshift=8.1cm,rotate around={45:(0,0)}]
\draw[thick, dotted] (8.5,4.5) -- (11,2) ;
\draw[thick, dotted] (8,4)node[left]{$\dots$} -- (10.5,1.5)node[right]{$S_2$};
\draw[thick, dotted] (7.5,3.5) -- (10,1);
\draw[ultra thick,->] (10.5,2.5) -- (11,2);
\draw[thick, dotted] (7.5,2.5) -- (9.5,4.5) ;
\draw[thick, dotted] (8,2) -- (10,4)node[above right]{$S_1$};
\draw[thick, dotted] (8.5,1.5) -- (10.5,3.5);
\draw[thick, dotted] (9,1) -- (11,3);
\draw[ultra thick,->] (10,3) -- (10.5,3.5); \draw[ultra thick,->] (9,4) -- (9.5,4.5);
\draw[rotate around={-45:(7.75,3)},fill=white] (7.75,3)rectangle({7.75+3.5*sqrt(2)/2},{3+2.5*sqrt(2)/2})node[midway]{$1$};
\end{scope}
\begin{scope}[xshift=3.6cm,rotate around={45:(0,0)}]
\draw[thick, dotted] (8.5,4.5) -- (11,2) ;
\draw[thick, dotted] (8,4) -- (10.5,1.5);
\draw[thick, dotted] (7.5,3.5) -- (10,1);
\draw[thick, dotted] (7.5,2.5) -- (9.5,4.5) ;
\draw[thick, dotted] (8,2) -- (10,4);
\draw[thick, dotted] (8.5,1.5) -- (10.5,3.5);
\draw[thick, dotted] (9,1) -- (11,3);
\draw[rotate around={-45:(7.75,3)},fill=white] (7.75,3)rectangle({7.75+3.5*sqrt(2)/2},{3+2.5*sqrt(2)/2})node[midway]{$u^{L-2}$};
\end{scope}
\begin{scope}[rotate around={45:(0,0)}]
\draw[thick, dotted] (8.5,4.5) -- (11,2) ;
\draw[thick, dotted] (8,4) -- (10.5,1.5);
\draw[thick, dotted] (7.5,3.5) -- (10,1);
\draw[ultra thick,->] (8.5,4.5) -- (8.8,4.2);
\draw[ultra thick,->] (8,4)node[left]{$T_2$} -- (8.3,3.7);
\draw[ultra thick,->] (7.5,3.5) -- (7.8,3.2);
\draw[thick, dotted] (7.5,2.5) -- (9.5,4.5) ;
\draw[thick, dotted] (8,2)node[below right]{$T_1$} -- (10,4);
\draw[thick, dotted] (8.5,1.5) -- (10.5,3.5);
\draw[thick, dotted] (9,1) -- (11,3);
\draw[rotate around={-45:(7.75,3)},fill=white] (7.75,3)rectangle({7.75+3.5*sqrt(2)/2},{3+2.5*sqrt(2)/2})node[midway]{$u^{L-1}$};
\end{scope}
\draw[ultra thick, gray](6.8,7.3)--(14.2,7.3);
\draw[ultra thick, gray]({5.3-sqrt(2)*3/2},{7.55+sqrt(2)*3/2})--({12.6-sqrt(2)*3/2},{7.55+sqrt(2)*3/2});
\end{tikzpicture}
\end{equation*}
where the index $l$, mentioned above, increases from right to left in this picture. We stress that the grey line indicates that the edges are identified in the periodic manner previously described, and are thus internal and implicitly summed over. Let us remark that while we have chosen to draw the diagrams so that the horizontal edges are connected, one could equally draw the pictures so that the vertical edges are connected; hence the rows and columns are treated symmetrically even though they are not drawn so.

We now wish to define the $L\to\infty$ limit of the model to obtain the quasi-periodic six vertex model. One must show that such a limit can indeed by taken, but in fact our matching with the periodic Hall--Littlewood process will show this as a biproduct. Thus, we will for now assume this to be true.

We then define the \emph{quasi-periodic stochastic six vertex model} as the limit of the quasi-periodic six vertex model of length $L$ when $T_1=(0,\dotsc, 0)$ has no arrows, $T_2=(1,\dotsc, 1)$ has arrows on each edge, and $L\to\infty$. By convention, we will use $S_1$ and $S_2$ to denote the locations of the outgoing arrows, and we will let $W$ denote the number of times arrows enter from the bottom of a six vertex domain, which we will refer to as the number of times arrows \emph{wind}. For a sequence of increasing partitions $\vec\lambda=(\lambda^{(0)}\subseteq\dotsm\subseteq\lambda^{(N)})$, we let $[\vec\lambda]\in\{0,1\}^N$ be the vector with a $1$ in position $i$ if $l(\lambda^{(i)})>l(\lambda^{(i-1)})$, and $0$ otherwise, and we let $[\vec\lambda]^c$ be defined similarly but with a $1$ if the length does not increase; i.e. $[\vec\lambda]^c = 1^N-[\vec\lambda]$. We are now ready to state the main result of this section, a matching between the periodic Hall--Littlewood process and this quasi-periodic six vertex model.

\begin{theorem}
\label{thm: HL 6vm}
Let $(\vec\lambda,\vec\mu)\sim \mathbb{PHL}^{a,b}_{t,u}$, and let $S_1$, $S_2$, and $W$ denote the outgoing arrows at the top and right, and number of times arrow wind in the quasi-periodic six vertex model, respectively (with the same parameters $u,t$ and rapidities $a,b$). Then for all $n\in \N$, $s_1\in\{0,1\}^M$, and $s_2\in\{0,1\}^N$, we have
\begin{equation*}
    \mathbb{PHL}_{t,u}^{a,b}(l(\lambda^{(0)})=n,[\vec\lambda]=s_1,[\vec\mu]^c=s_2)=\PP(W+\chi=n,S_1=s_1,S_2=s_2),
\end{equation*}
where $\chi$ is an independent $u$-geometric random variable.
\end{theorem}
\begin{remark}
As already noted, we have not proved that the limiting procedure to obtain the quasi-periodic six vertex model actually converges. One should interpret the theorem as saying that the observables of the quasi-periodic six vertex model of length $L$ converges in distribution to those of the periodic Hall--Littlewood process as $L\to\infty$. Convergence of these observables is actually enough to define the entire infinite quasi-periodic model, as one can then define a sequence of growing finite portions of the model in a consistent manner.
\end{remark}

\begin{remark}
\label{rmk: down right path}
One could state and prove a more general statement, with a periodic Hall--Littlewood process with multiple ascending/descending specializations, and the quasi-periodic six vertex model where the output is a general down right domain, in the same spirit as Theorem 5.6 of \cite{BBW16}. The proof is essentially the same, and so for the sake of simplifying the notation and diagrams we have only stated the simplest case. 

When $u=0$, the quasi-periodic six vertex model reduces to the usual six vertex model in a rectangle, and the periodic Hall--Littlewood measure becomes the usual Hall--Littlewood measure, and Theorem \ref{thm: HL 6vm} reduces to Theorem 4.3 of \cite{BBW16}.
\end{remark}

\subsection{Deformed bosons}
We now introduce an additional vertex model which is needed in the proof of Theorem \ref{thm: HL 6vm}. It is a model for arrows on a lattice, where the horizontal edges again can have at most one arrow, but where the vertical edges can have any number of arrows. A local parameter (called a \emph{rapidity}) is associated to each row of the lattice, while all vertices depend on a global parameter $t$. The vertex weights are given by

\begin{align}
\label{eq:black-vertices}
\begin{array}{cccc}
\begin{tikzpicture}[scale=0.8,>=stealth]
\draw[lgray,ultra thick] (-1,0) -- (1,0);
\draw[lgray,line width=10pt] (0,-1) -- (0,1);
\node[below] at (0,-1) {$m$};
\draw[ultra thick,->,rounded corners] (-0.075,-1) -- (-0.075,1);
\draw[ultra thick,->,rounded corners] (0.075,-1) -- (0.075,1);
\node[above] at (0,1) {$m$};
\end{tikzpicture}
\quad\quad\quad
&
\begin{tikzpicture}[scale=0.8,>=stealth]
\draw[lgray,ultra thick] (-1,0) -- (1,0);
\draw[lgray,line width=10pt] (0,-1) -- (0,1);
\node[below] at (0,-1) {$m$};
\draw[ultra thick,->,rounded corners] (-0.075,-1) -- (-0.075,1);
\draw[ultra thick,->,rounded corners] (0.075,-1) -- (0.075,0) -- (1,0);
\node[above] at (0,1) {$m-1$};
\end{tikzpicture}
\quad\quad\quad
&
\begin{tikzpicture}[scale=0.8,>=stealth]
\draw[lgray,ultra thick] (-1,0) -- (1,0);
\draw[lgray,line width=10pt] (0,-1) -- (0,1);
\node[below] at (0,-1) {$m$};
\draw[ultra thick,->,rounded corners] (-1,0) -- (-0.15,0) -- (-0.15,1);
\draw[ultra thick,->,rounded corners] (0,-1) -- (0,1);
\draw[ultra thick,->,rounded corners] (0.15,-1) -- (0.15,1);
\node[above] at (0,1) {$m+1$};
\end{tikzpicture}
\quad\quad\quad
&
\begin{tikzpicture}[scale=0.8,>=stealth]
\draw[lgray,ultra thick] (-1,0) -- (1,0);
\draw[lgray,line width=10pt] (0,-1) -- (0,1);
\node[below] at (0,-1) {$m$};
\draw[ultra thick,->,rounded corners] (-1,0) -- (-0.15,0) -- (-0.15,1);
\draw[ultra thick,->,rounded corners] (0,-1) -- (0,1);
\draw[ultra thick,->,rounded corners] (0.15,-1) -- (0.15,0) -- (1,0);
\node[above] at (0,1) {$m$};
\end{tikzpicture}
\\
1
\quad\quad\quad
&
a
\quad\quad\quad
&
(1-t^{m+1})
\quad\quad\quad
&
a
\end{array}\end{align}
where $0\leq a<1$ is the row rapidity, and $m$ is the number of arrows entering from below.   

We will also need another version of this model with an alternative normalization. It is defined in the same way, but with alternative vertex weights

\begin{align}
\label{eq:red-vertices}
\begin{array}{cccc}
\begin{tikzpicture}[scale=0.8,>=stealth]
\draw[lred, ultra thick] (-1,0) -- (1,0);
\draw[lred,line width=10pt] (0,-1) -- (0,1);
\node[below] at (0,-1) {$m$};
\draw[ultra thick,->,rounded corners] (-0.075,-1) -- (-0.075,1);
\draw[ultra thick,->,rounded corners] (0.075,-1) -- (0.075,1);
\node[above] at (0,1) {$m$};
\end{tikzpicture}
\quad\quad\quad
&
\begin{tikzpicture}[scale=0.8,>=stealth]
\draw[lred,ultra thick] (-1,0) -- (1,0);
\draw[lred,line width=10pt] (0,-1) -- (0,1);
\node[below] at (0,-1) {$m$};
\draw[ultra thick,->,rounded corners] (-0.075,-1) -- (-0.075,1);
\draw[ultra thick,->,rounded corners] (0.075,-1) -- (0.075,0) -- (1,0);
\node[above] at (0,1) {$m-1$};
\end{tikzpicture}
\quad\quad\quad
&
\begin{tikzpicture}[scale=0.8,>=stealth]
\draw[lred,ultra thick] (-1,0) -- (1,0);
\draw[lred,line width=10pt] (0,-1) -- (0,1);
\node[below] at (0,-1) {$m$};
\draw[ultra thick,->,rounded corners] (-1,0) -- (-0.15,0) -- (-0.15,1);
\draw[ultra thick,->,rounded corners] (0,-1) -- (0,1);
\draw[ultra thick,->,rounded corners] (0.15,-1) -- (0.15,1);
\node[above] at (0,1) {$m+1$};
\end{tikzpicture}
\quad\quad\quad
&
\begin{tikzpicture}[scale=0.8,>=stealth]
\draw[lred,ultra thick] (-1,0) -- (1,0);
\draw[lred,line width=10pt] (0,-1) -- (0,1);
\node[below] at (0,-1) {$m$};
\draw[ultra thick,->,rounded corners] (-1,0) -- (-0.15,0) -- (-0.15,1);
\draw[ultra thick,->,rounded corners] (0,-1) -- (0,1);
\draw[ultra thick,->,rounded corners] (0.15,-1) -- (0.15,0) -- (1,0);
\node[above] at (0,1) {$m$};
\end{tikzpicture}
\\
b
\quad\quad\quad
&
1
\quad\quad\quad
&
b (1-t^{m+1})
\quad\quad\quad
&
1
\end{array}
\end{align}
where again, $0\leq b<1$ is the row rapidity and $m$ is the number of arrows entering from below.

In both cases, we will use a graphical notation to write partition functions. For a single row, we will write $w_a(\cdot )$ around a picture to represent the sum over all internal edges of the products of the vertex weights, with external edges fixed and $a$ denoting the row rapidity. If there is more than one row, we will instead simply give the picture with row rapidities identified.

We shall also wish to consider infinite rows, formally defined as a limit of progressively longer finite rows. For finitely supported sequences of non-negative integers $m_i$ and $n_i$, we will let
\begin{multline}
\label{eq:L-limit}
    w_a
\left(
\begin{tikzpicture}[baseline=(current bounding box.center),>=stealth,scale=0.8]
\draw[lgray,ultra thick] (0,0) -- (7,0);
\foreach\x in {1,...,6}{
\draw[lgray,line width=10pt] (7-\x,-1) -- (7-\x,1);
}
\node[left] at (0,0) {$0$};
\node[right] at (7,0) {$j$};
\foreach\x in {1,2,3}{
\node[text centered,below] at (7-\x,-1) {\tiny $m_{\x}$};
\node[text centered,above] at (7-\x,1) {\tiny $n_{\x}$};
}
\foreach\x in {4,5}{
\node[text centered,below] at (7-\x,-1) {\tiny $\cdots$};
\node[text centered,above] at (7-\x,1) {\tiny $\cdots$};
}
\end{tikzpicture}\right)
\\=\lim_{N\to\infty}w_a
\left(
\begin{tikzpicture}[baseline=(current bounding box.center),>=stealth,scale=0.8]
\draw[lgray,ultra thick] (0,0) -- (7,0);
\foreach\x in {1,...,6}{
\draw[lgray,line width=10pt] (7-\x,-1) -- (7-\x,1);
}
\node[left] at (0,0) {$0$};
\node[right] at (7,0) {$j$};
\foreach\x in {1,2,3}{
\node[text centered,below] at (7-\x,-1) {\tiny $m_{\x}$};
\node[text centered,above] at (7-\x,1) {\tiny $n_{\x}$};
}
\foreach\x in {4,5}{
\node[text centered,below] at (7-\x,-1) {\tiny $\cdots$};
\node[text centered,above] at (7-\x,1) {\tiny $\cdots$};
}
\node[text centered,below] at (1,-1) {\tiny $m_{N}$};
\node[text centered,above] at (1,1) {\tiny $n_{N}$};
\end{tikzpicture}\right),
\end{multline}
and similarly for the second set of weights \eqref{eq:red-vertices}, except with an arrow entering from the left. These limits are well-defined, because the $m_i$ and $n_i$ are finitely supported, so eventually the weights will all be $1$ sufficiently far to the left in the infinite product of vertices. Note that if an arrow entered from the left in \eqref{eq:L-limit} (or no arrow entered from the left in the case of rows constructed from the weights \eqref{eq:red-vertices}), the weights would eventually all be $a$, and not $1$, and the limit would simply be $0$ since $a<1$.

The reason that this model is useful to prove Theorem \ref{thm: HL 6vm} is that it shares the same $R$-matrix as the six vertex model, but its partition functions are given by Hall--Littlewood polynomials. These two statements are given in the following proposition and lemma. We will use a cross rotated by $45^\circ$ to denote a \emph{Yang--Baxter vertex}; this is a vertex with the same weights as those of the six--vertex model (cf. Figure \ref{fig:6v}) with corresponding row and column parameter.

\begin{proposition}[{\hspace{1sp}\cite[Proposition 4.8]{BBCW18}}]
\label{prop: YB boson}
For any finitely supported sequences $n_i$ and $m_i$, and any $j_1,j_2 \in \{0,1\}$, we have
\begin{equation*}
\label{graph-exchange}
\left(
\frac{1-a b}{1-t a b}
\right)
\sum_{p_i}
\begin{tikzpicture}[baseline=(current bounding box.center),>=stealth,scale=0.7]
\draw[lgray,ultra thick] (-1,1) node[left,black] {$a$}
-- (4,1) node[right,black] {$j_2$};
\draw[lred,ultra thick] (-1,0) node[left,black] {$b$}
-- (4,0) node[right,black] {$j_1$};
\foreach\x in {0,...,3}{
\draw[lgray,line width=10pt] (3-\x,0.5) -- (3-\x,2);
\draw[lred,line width=10pt] (3-\x,-1) -- (3-\x,0.5);
}
\node[below] at (3,-1) {$m_1$};
\node at (3,0.5) {$p_1$};
\node[above] at (3,2) {$n_1$};
\node[below] at (2,-1) {$m_2$};
\node at (2,0.5) {$p_2$};
\node[above] at (2,2) {$n_2$};
\node[text centered] at (0,0.5) {$\cdots$};
\node[text centered] at (1,0.5) {$\cdots$};
\draw[ultra thick,->] (-1,0) -- (0,0);
\end{tikzpicture}
=
\sum_{p_i,k_1,k_2}
\begin{tikzpicture}[baseline=(current bounding box.center),>=stealth,scale=0.7]
\foreach\x in {0,...,3}{
\draw[lgray,line width=10pt] (3-\x,-1) -- (3-\x,0.5);
\draw[lred,line width=10pt] (3-\x,0.5) -- (3-\x,2);
}
\draw[lred,ultra thick] (-1,1) node[left,black] {$b$}
-- (4,1);
\draw[lgray,ultra thick] (-1,0) node[left,black] {$a$}
-- (4,0);
\draw[dotted,thick] (4,1) node[above] {$k_1$} -- (5,0) node[right] {$j_1$};
\draw[dotted,thick] (4,0) node[below] {$k_2$} -- (5,1) node[right] {$j_2$};
\node[below] at (3,-1) {$m_1$};
\node at (3,0.5) {$p_1$};
\node[above] at (3,2) {$n_1$};
\node[below] at (2,-1) {$m_2$};
\node at (2,0.5) {$p_2$};
\node[above] at (2,2) {$n_2$};
\node[text centered] at (0,0.5) {$\cdots$};
\node[text centered] at (1,0.5) {$\cdots$};
\draw[ultra thick,->] (-1,1) -- (0,1);
\end{tikzpicture},
\end{equation*}
where on the left, $a$ and $b$ indicate the row rapidities, and the left boundary conditions are given by the arrows as indicated.
\end{proposition}

\begin{lemma}[{\hspace{1sp}\cite[Lemma 4.11]{BBCW18}}]
\label{lem: boson HL}
Fix two partitions $\lambda$ and $\mu$ such that $\lambda=1^{m_1(\lambda)} 2^{m_2(\lambda)} \ldots$ and $\mu=1^{m_1(\mu)} 2^{m_2(\mu)} \ldots$. We then have
\begingroup
\allowdisplaybreaks
\begin{align*}
    w_a
\left(
\begin{tikzpicture}[baseline=(current bounding box.center),>=stealth,scale=0.8]
\draw[lgray,ultra thick] (0,0) -- (7,0);
\foreach\x in {1,...,6}{
\draw[lgray,line width=10pt] (7-\x,-1) -- (7-\x,1);
}
\node[left] at (0,0) {$0$};
\node[right] at (7,0) {$0$};
\foreach\x in {1,2,3}{
\node[text centered,below] at (7-\x,-1) {\tiny $m_{\x}(\lambda)$};
\node[text centered,above] at (7-\x,1) {\tiny $m_{\x}(\mu)$};
}
\foreach\x in {4,5}{
\node[text centered,below] at (7-\x,-1) {\tiny $\cdots$};
\node[text centered,above] at (7-\x,1) {\tiny $\cdots$};
}
\end{tikzpicture}
\right)&=\mathbf{1}_{l(\lambda)=l(\mu)}P_{\lambda/\mu}(a;0,t),
    \\w_a
\left(
\begin{tikzpicture}[baseline=(current bounding box.center),>=stealth,scale=0.8]
\draw[lgray,ultra thick] (0,0) -- (7,0);
\foreach\x in {1,...,6}{
\draw[lgray,line width=10pt] (7-\x,-1) -- (7-\x,1);
}
\node[left] at (0,0) {$0$};
\draw[ultra thick,->] (6,0) -- (7,0);
\node[right] at (7,0) {$1$};
\foreach\x in {1,2,3}{
\node[text centered,below] at (7-\x,-1) {\tiny $m_{\x}(\lambda)$};
\node[text centered,above] at (7-\x,1) {\tiny $m_{\x}(\mu)$};
}
\foreach\x in {4,5}{
\node[text centered,below] at (7-\x,-1) {\tiny $\cdots$};
\node[text centered,above] at (7-\x,1) {\tiny $\cdots$};
}
\end{tikzpicture}\right)&=\mathbf{1}_{l(\lambda)=l(\mu)+1}P_{\lambda/\mu}(a;0,t),
    \\w_a
\left(
\begin{tikzpicture}[baseline=(current bounding box.center),>=stealth,scale=0.8]
\draw[lred,ultra thick] (0,0) -- (7,0);
\foreach\x in {1,...,6}{
\draw[lred,line width=10pt] (7-\x,-1) -- (7-\x,1);
}
\node[left] at (0,0) {$1$};
\draw[ultra thick,->] (0,0) -- (1,0);
\node[right] at (7,0) {$0$};
\foreach\x in {1,2,3}{
\node[text centered,below] at (7-\x,-1) {\tiny $m_{\x}(\mu)$};
\node[text centered,above] at (7-\x,1) {\tiny $m_{\x}(\lambda)$};
}
\foreach\x in {4,5}{
\node[text centered,below] at (7-\x,-1) {\tiny $\cdots$};
\node[text centered,above] at (7-\x,1) {\tiny $\cdots$};
}
\end{tikzpicture}\right)&=\mathbf{1}_{l(\lambda)=l(\mu)+1}Q_{\lambda/\mu}(a;0,t),
    \\w_a
\left(
\begin{tikzpicture}[baseline=(current bounding box.center),>=stealth,scale=0.8]
\draw[lred,ultra thick] (0,0) -- (7,0);
\foreach\x in {1,...,6}{
\draw[lred,line width=10pt] (7-\x,-1) -- (7-\x,1);
}
\node[left] at (0,0) {$1$};
\draw[ultra thick,->] (0,0) -- (1,0);
\draw[ultra thick,->] (6,0) -- (7,0);
\node[right] at (7,0) {$1$};
\foreach\x in {1,2,3}{
\node[text centered,below] at (7-\x,-1) {\tiny $m_{\x}(\mu)$};
\node[text centered,above] at (7-\x,1) {\tiny $m_{\x}(\lambda)$};
}
\foreach\x in {4,5}{
\node[text centered,below] at (7-\x,-1) {\tiny $\cdots$};
\node[text centered,above] at (7-\x,1) {\tiny $\cdots$};
}
\end{tikzpicture}\right)&=\mathbf{1}_{l(\lambda)=l(\mu)}Q_{\lambda/\mu}(a;0,t),
\end{align*}
\endgroup
where the right hand sides of these expressions denote skew Hall--Littlewood polynomials.
\end{lemma}

We also need the following lemma which allows powers of $u$ to be absorbed into the parameters $a$ and $b$.

\begin{lemma}
\label{lem: u power shift}
Let $n_1,\dotsc$ and $m_1,\dotsc$ be finitely supported sequences of non-negative integers. We have the following equalities of one-row partition functions:
\begin{multline*}
    u^{\sum_i in_i} w_{ua}
\left(
\begin{tikzpicture}[baseline=(current bounding box.center),>=stealth,scale=0.8]
\draw[lgray,ultra thick] (0,0) -- (7,0);
\foreach\x in {1,...,6}{
\draw[lgray,line width=10pt] (7-\x,-1) -- (7-\x,1);
}
\node[left] at (0,0) {$0$};
\node[right] at (7,0) {$j$};
\foreach\x in {1,2,3}{
\node[text centered,below] at (7-\x,-1) {\tiny $m_{\x}$};
\node[text centered,above] at (7-\x,1) {\tiny $n_{\x}$};
}
\foreach\x in {4,5}{
\node[text centered,below] at (7-\x,-1) {\tiny $\cdots$};
\node[text centered,above] at (7-\x,1) {\tiny $\cdots$};
}
\end{tikzpicture}\right)
\\
=u^{\sum_i im_i}w_a
\left(
\begin{tikzpicture}[baseline=(current bounding box.center),>=stealth,scale=0.8]
\draw[lgray,ultra thick] (0,0) -- (7,0);
\foreach\x in {1,...,6}{
\draw[lgray,line width=10pt] (7-\x,-1) -- (7-\x,1);
}
\node[left] at (0,0) {$0$};
\node[right] at (7,0) {$j$};
\foreach\x in {1,2,3}{
\node[text centered,below] at (7-\x,-1) {\tiny $m_{\x}$};
\node[text centered,above] at (7-\x,1) {\tiny $n_{\x}$};
}
\foreach\x in {4,5}{
\node[text centered,below] at (7-\x,-1) {\tiny $\cdots$};
\node[text centered,above] at (7-\x,1) {\tiny $\cdots$};
}
\end{tikzpicture}\right),
\end{multline*}
and
\begin{multline*}
    u^{\sum_i in_i} w_{a}
\left(
\begin{tikzpicture}[baseline=(current bounding box.center),>=stealth,scale=0.8]
\draw[lred,ultra thick] (0,0) -- (7,0);
\foreach\x in {1,...,6}{
\draw[lred,line width=10pt] (7-\x,-1) -- (7-\x,1);
}
\node[left] at (0,0) {$1$};
\node[right] at (7,0) {$j$};
\foreach\x in {1,2,3}{
\node[text centered,below] at (7-\x,-1) {\tiny $m_{\x}$};
\node[text centered,above] at (7-\x,1) {\tiny $n_{\x}$};
}
\foreach\x in {4,5}{
\node[text centered,below] at (7-\x,-1) {\tiny $\cdots$};
\node[text centered,above] at (7-\x,1) {\tiny $\cdots$};
\draw[ultra thick,->] (0,0) -- (1,0);
}
\end{tikzpicture}\right)
\\
=u^{\sum_i im_i}w_{ua}
\left(
\begin{tikzpicture}[baseline=(current bounding box.center),>=stealth,scale=0.8]
\draw[lred,ultra thick] (0,0) -- (7,0);
\foreach\x in {1,...,6}{
\draw[lred,line width=10pt] (7-\x,-1) -- (7-\x,1);
}
\node[left] at (0,0) {$1$};
\node[right] at (7,0) {$j$};
\foreach\x in {1,2,3}{
\node[text centered,below] at (7-\x,-1) {\tiny $m_{\x}$};
\node[text centered,above] at (7-\x,1) {\tiny $n_{\x}$};
}
\foreach\x in {4,5}{
\node[text centered,below] at (7-\x,-1) {\tiny $\cdots$};
\node[text centered,above] at (7-\x,1) {\tiny $\cdots$};
\draw[ultra thick,->] (0,0) -- (1,0);
}
\end{tikzpicture}\right).
\end{multline*}
\end{lemma}
\begin{proof}
The first equality comes from the fact that the differences in the powers of $u$ comes from arrows travelling right, and this is compensated by the fact that each time an arrow travels right, it picks up a power of $u$ due to the parameter on the left hand side being $ua$.

The second equality is similar, except that there is an extra arrow entering from the left. This adds an extra power of $u$ which is not counted on the right hand side, and subtracting the contribution coming from arrows traveling right we obtain a power of $u$ each time there is no arrow travelling horizontally, which is accounted for by the power of $u$ due to the parameter on the right hand side being $ua$.
\end{proof}

\subsection{Proof of Theorem \ref{thm: HL 6vm}}
With the deformed boson model, we are now able to give a proof of Theorem \ref{thm: HL 6vm} in the same spirit as that of \cite{BBW16}.

\begin{proof}[Proof of Theorem \ref{thm: HL 6vm}]
We use Lemma \ref{lem: boson HL} and the definition of the periodic Hall--Littlewood process (see equation \eqref{eq:periodic-macdonald} with $q=0$) to write

\begin{multline}
\label{eq:rb-lattice}
\mathbb{PHL}^{a,b}_{t,u}(l(\lambda^{(0)})=n,[\vec\lambda]=S_1,[\vec\mu]^c=S_2)
\\=
\frac{1}{\Phi(a,b;0,t,u)}
\times\sum_{\mu,\lambda}u^{|\mu|}
\begin{tikzpicture}[scale=0.8,baseline=(current bounding box.center),>=stealth]
\foreach\x in {0,...,6}{
\draw[lgray,line width=10pt] (\x,2) -- (\x,7);
}
\foreach\y in {3,...,6}{
\draw[lgray,thick] (-1,\y) -- (7,\y);
}
\foreach\x in {0,...,6}{
\draw[lred,line width=10pt] (\x,-2) -- (\x,2);
}
\foreach\y in {-1,...,1}{
\draw[lred,thick] (-1,\y) -- (7,\y);
}
\draw[ultra thick,->] (-1,-1) -- (0,-1); \draw[ultra thick,->] (-1,0) -- (0,0); \draw[ultra thick,->] (-1,1) -- (0,1);
\draw[ultra thick,->] (6,1) -- (7,1); \draw[ultra thick,->] (6,4) -- (7,4); \draw[ultra thick,->] (6,6) -- (7,6);
\node[right] at (7,-1) {$S_2(1)$}; 
\node[right] at (7,0.2) {$\vdots$}; 
\node[right] at (7,1) {$S_2(N)$};
\node[right] at (7,3) {$S_1(M)$};
\node[right] at (7,4.5) {$\vdots$};
\node[right] at (7,6) {$S_1(1)$};
\node[left] at (-1,6) {$a_1$};
\node[left] at (-1,4.5) {$\vdots$};
\node[left] at (-1,3) {$a_M$};
\node[left] at (-1,1) {$b_N$};
\node[left] at (-1,0.2) {$\vdots$};
\node[left] at (-1,-1) {$b_1$};
\node[below] at (6,-2) {\tiny $m_1(\mu)$};
\node[below] at (5,-2) {\tiny $m_2(\mu)$};
\node[below] at (4,-2) {\tiny $m_3(\mu)$};
\node[below] at (6,7.75) {\tiny $m_1(\mu)$};
\node[below] at (5,7.75) {\tiny $m_2(\mu)$};
\node[below] at (4,7.75) {\tiny $m_3(\mu)$};
\node[below] at (6,2.5) {\tiny $m_1(\lambda)$};
\node[below] at (5,2.5) {\tiny $m_2(\lambda)$};
\node[below] at (4,2.5) {\tiny $m_3(\lambda)$};
\node[below] at (1,7.75) {$\longleftarrow$};
\node[below] at (3,7.75) {$\cdots$};
\node[below] at (0,7.75) {$\infty$};
\node[below] at (1,-2) {$\longleftarrow$};
\node[below] at (3,-2) {$\cdots$};
\node[below] at (0,-2) {$\infty$};
\end{tikzpicture}
\end{multline}
Here, the sum is over all $\lambda$ and $\mu$ such that $l(\mu)=n$, which can equivalently be thought of as counting the number of arrows entering the bottom, which is equal to the number exiting the top. This picture can be viewed as lying on a cylinder, with the top vertical edges being identified with those at the bottom. The $u^{|\mu|}$ factor penalizes windings, thus ensuring convergence.

We then use Proposition \ref{prop: YB boson} iteratively to swap the red and gray portions of the deformed boson model, introducing an $M\times N$ rectangle of Yang--Baxter vertices to the right, and obtaining
\begin{multline}
\label{eq:first-iteration}
\mathbb{PHL}^{a,b}_{t,u}(l(\lambda^{(0)})=n,[\vec\lambda]=S_1,[\vec\mu]^c=S_2)
=
\frac{1}{\Phi(a,b;0,t,u)}\prod_{i,j}\frac{1-ta_ib_j}{1-a_ib_j}\times\\\sum_{\mu,\lambda}u^{|\mu|}
\begin{tikzpicture}[scale=0.8,baseline=(current bounding box.center),>=stealth]
\foreach\x in {0,...,6}{
\draw[lgray,line width=10pt] (\x,-2) -- (\x,3);
}
\foreach\y in {-1,...,2}{
\draw[lgray,thick] (-1,\y) -- (7,\y);
}
\foreach\x in {0,...,6}{
\draw[lred,line width=10pt] (\x,3) -- (\x,7);
}
\foreach\y in {4,...,6}{
\draw[lred,thick] (-1,\y) -- (7,\y);
}
\draw[ultra thick, gray] (7,3)--(9.5,0.5) node[midway, below left ,black]{$W_1$};
\draw[thick, dotted] (7,6) -- (11,2) node[below right] {\tiny $S_2(N)$};
\draw[thick, dotted] (7,5) -- (10.5,1.5);
\draw (10.8,1.35)node{\reflectbox{$\ddots$}};
\draw[thick, dotted] (7,4) -- (10,1) node[below right] {\tiny $S_2(1)$};
\draw[ultra thick,->] (10.5,2.5) -- (11,2);
\draw[thick, dotted] (7,2) -- (9.5,4.5) node[above right] {\tiny $S_1(1)$};
\draw[thick, dotted] (7,1) -- (10,4);
\draw[thick, dotted] (7,0) -- (10.5,3.5);
\draw (10.25,3.75)node[above right]{$\ddots$};
\draw[thick, dotted] (7,-1) -- (11,3) node[above right] {\tiny $S_1(M)$};
\draw[ultra thick,->] (10,3) -- (10.5,3.5); \draw[ultra thick,->] (9,4) -- (9.5,4.5);
\draw[rotate around={-45:(7.75,3)},fill=white] (7.75,3)rectangle({7.75+3.5*sqrt(2)/2},{3+2.5*sqrt(2)/2})node[midway]{$1$};
\draw[ultra thick,->] (-1,6) -- (0,6); \draw[ultra thick,->] (-1,5) -- (0,5); \draw[ultra thick,->] (-1,4) -- (0,4);
\node[left] at (-1,2) {$a_1$};
\node[left] at (-1,0.5) {$\vdots$};
\node[left] at (-1,-1) {$a_M$};
\node[left] at (-1,6) {$b_N$};
\node[left] at (-1,5.2) {$\vdots$};
\node[left] at (-1,4) {$b_1$};
\node[below] at (6,-2) {\tiny $m_1(\mu)$};
\node[below] at (5,-2) {\tiny $m_2(\mu)$};
\node[below] at (4,-2) {\tiny $m_3(\mu)$};
\node[below] at (6,7.75) {\tiny $m_1(\mu)$};
\node[below] at (5,7.75) {\tiny $m_2(\mu)$};
\node[below] at (4,7.75) {\tiny $m_3(\mu)$};
\node[below] at (6,3) {\tiny $m_1(\lambda)$};
\node[below] at (5,3) {\tiny $m_2(\lambda)$};
\node[below] at (4,3) {\tiny $m_3(\lambda)$};
\node[below] at (1,7.75) {$\longleftarrow$};
\node[below] at (3,7.75) {$\cdots$};
\node[below] at (0,7.75) {$\infty$};
\node[below] at (1,-2) {$\longleftarrow$};
\node[below] at (3,-2) {$\cdots$};
\node[below] at (0,-2) {$\infty$};
\end{tikzpicture}
\end{multline}

The rectangle on the right represents the six vertex model on an $M\times N$ domain, with parameter $z=1$. We now use Lemma \ref{lem: u power shift} to change the $u^{|\mu|}$ factor into $u^{|\lambda|}$, at the cost of multiplying all the $a_i$ by a factor of $u$. Moreover, we can replace the constraint that $l(\mu)=n$ with the constraint that $l(\lambda)+W_1=n$, where $W_1$ is the number of arrows entering the six vertex model from the bottom (in the diagram, this is the number of arrows crossing the gray line indicated) using conservation of the number of arrows entering and exiting the grey portion of the deformed boson model. Then we can exchange the roles of $\lambda$ and $\mu$, and after shifting the diagram to start from the bottom of the red portion of the lattice (or alternatively, simply rotating the cylinder on which this diagram is drawn), we obtain the expression
\begin{multline}
\mathbb{PHL}^{a,b}_{t,u}(l(\lambda^{(0)})=n,[\vec\lambda]=S_1,[\vec\mu]^c=S_2)
=
\frac{1}{\Phi(a,b;0,t,u)}\prod_{i,j}\frac{1-ta_ib_j}{1-a_ib_j}
\times
\\\sum_{\mu,\lambda}u^{|\mu|}
\begin{tikzpicture}[scale=0.8,baseline=(current bounding box.center),>=stealth]
\foreach\x in {0,...,6}{
\draw[lgray,line width=10pt] (\x,2) -- (\x,7);
}
\foreach\y in {3,...,6}{
\draw[lgray,thick] (-1,\y) -- (7,\y);
}
\foreach\x in {0,...,6}{
\draw[lred,line width=10pt] (\x,-2) -- (\x,2);
}
\foreach\y in {-1,...,1}{
\draw[lred,thick] (-1,\y) -- (7,\y);
}
\draw[ultra thick,->] (-1,-1) -- (0,-1); \draw[ultra thick,->] (-1,0) -- (0,0); \draw[ultra thick,->] (-1,1) -- (0,1);
\draw[thick, dotted] (7,6) -- (7.5,6.5);
\draw[thick, dotted] (7,5) -- (8,6);
\draw[thick, dotted] (7,4) -- (8.5,5.5);
\draw[thick, dotted] (7,3) -- (9,5);
\draw[ultra thick, gray] (6.75,6.75)--(9.25,4.25);
\begin{scope}[yshift=-5cm]
\draw[ultra thick, gray] (7,3)--(9.5,0.5) node[midway, below left ,black]{$W_1$};
\draw[thick, dotted] (7,6) -- (11,2) node[below right] {\tiny $S_2(N)$};
\draw[thick, dotted] (7,5) -- (10.5,1.5);
\draw (10.8,1.35)node{\reflectbox{$\ddots$}};
\draw[thick, dotted] (7,4) -- (10,1) node[below right] {\tiny $S_2(1)$};
\draw[ultra thick,->] (10.5,2.5) -- (11,2);
\draw[thick, dotted] (7,2) -- (9.5,4.5) node[above right] {\tiny $S_1(1)$};
\draw[thick, dotted] (7.5,1.5) -- (10,4);
\draw[thick, dotted] (8,1) -- (10.5,3.5);
\draw (10.25,3.75)node[above right]{$\ddots$};
\draw[thick, dotted] (8.5,0.5) -- (11,3) node[above right] {\tiny $S_1(M)$};
\draw[ultra thick,->] (10,3) -- (10.5,3.5); \draw[ultra thick,->] (9,4) -- (9.5,4.5);
\draw[rotate around={-45:(7.75,3)},fill=white] (7.75,3)rectangle({7.75+3.5*sqrt(2)/2},{3+2.5*sqrt(2)/2})node[midway]{$1$};
\end{scope}
\node[left] at (-1,6) {$ua_1$};
\node[left] at (-1,4.5) {$\vdots$};
\node[left] at (-1,3) {$ua_M$};
\node[left] at (-1,1) {$b_N$};
\node[left] at (-1,0.2) {$\vdots$};
\node[left] at (-1,-1) {$b_1$};
\node[below] at (6,-2) {\tiny $m_1(\mu)$};
\node[below] at (5,-2) {\tiny $m_2(\mu)$};
\node[below] at (4,-2) {\tiny $m_3(\mu)$};
\node[below] at (6,7.75) {\tiny $m_1(\mu)$};
\node[below] at (5,7.75) {\tiny $m_2(\mu)$};
\node[below] at (4,7.75) {\tiny $m_3(\mu)$};
\node[below] at (6,2.5) {\tiny $m_1(\lambda)$};
\node[below] at (5,2.5) {\tiny $m_2(\lambda)$};
\node[below] at (4,2.5) {\tiny $m_3(\lambda)$};
\node[below] at (1,7.75) {$\longleftarrow$};
\node[below] at (3,7.75) {$\cdots$};
\node[below] at (0,7.75) {$\infty$};
\node[below] at (1,-2) {$\longleftarrow$};
\node[below] at (3,-2) {$\cdots$};
\node[below] at (0,-2) {$\infty$};
\end{tikzpicture}
\end{multline}
The arrows exiting from the right of the grey portion of the lattice enter the six vertex model from the bottom; this is indicated by the grey line which shows the identified edges. At this point, we recognize that the bosonic portion of the diagram is identical to the one we started with, except that the $a_i$ have all been multiplied by $u$. We can thus iterate this procedure $L$ times, obtaining

\begin{multline}
\label{eq:final-bosonic-lattice}
\mathbb{PHL}^{a,b}_{t,u}(l(\lambda^{(0)})=n,[\vec\lambda]=S_1,[\vec\mu]^c=S_2)
=
\frac{1}{\Phi(a,b;0,t,u)}\prod_{l=0}^{L-1}\prod_{i,j}\frac{1-tu^la_ib_j}{1-u^la_ib_j}
\times
\\\sum_{\mu,\lambda}u^{|\mu|}
\begin{tikzpicture}[scale=0.8,baseline=(current bounding box.center),>=stealth]
\foreach\x in {0,...,6}{
\draw[lgray,line width=10pt] (\x,2) -- (\x,7);
}
\foreach\y in {3,...,6}{
\draw[lgray,thick] (-1,\y) -- (7,\y);
}
\foreach\x in {0,...,6}{
\draw[lred,line width=10pt] (\x,-2) -- (\x,2);
}
\foreach\y in {-1,...,1}{
\draw[lred,thick] (-1,\y) -- (7,\y);
}
\draw[ultra thick,->] (-1,-1) -- (0,-1); \draw[ultra thick,->] (-1,0) -- (0,0); \draw[ultra thick,->] (-1,1) -- (0,1);
\draw[thick, dotted] (7,6) -- (7.5,6.5);
\draw[thick, dotted] (7,5) -- (8,6);
\draw[thick, dotted] (7,4) -- (8.5,5.5);
\draw[thick, dotted] (7,3) -- (9,5);
\draw[ultra thick, gray] (6.75,6.75)--(9.25,4.25);
\begin{scope}[yshift=-5cm]
\draw[ultra thick, gray] (7,3)--(9.5,0.5) node[midway, below left ,black]{$W_L$};
\draw[thick, dotted] (7,6) -- (11,2);
\draw[thick, dotted] (7,5) -- (10.5,1.5);
\draw[thick, dotted] (7,4) -- (10,1) ;
\draw (11,1)node[]{$\ddots$};
\draw[thick, dotted] (7,2) -- (9.5,4.5) ;
\draw[thick, dotted] (7.5,1.5) -- (10,4);
\draw[thick, dotted] (8,1) -- (10.5,3.5);
\draw[thick, dotted] (8.5,0.5) -- (11,3) ;
\draw[rotate around={-45:(7.75,3)},fill=white] (7.75,3)rectangle({7.75+3.5*sqrt(2)/2},{3+2.5*sqrt(2)/2})node[midway]{$u^{L-1}$};
\end{scope}
\begin{scope}[xshift=3.5cm,yshift=-8.5cm]
\draw[ultra thick, gray] (7,3)--(9.5,0.5) node[midway, below left ,black]{$W_1$};
\draw[ultra thick, gray] (5.5,4.5)--(7.5,2.5);
\draw[ultra thick, gray, out=-45, in=135] (5.75,12.75) to (5.25,8.25)to(8.25,5.25);
\draw[thick, dotted] (8.5,4.5) -- (11,2) node[below right] {\tiny $S_2(N)$};
\draw[thick, dotted] (8,4) -- (10.5,1.5);
\draw (10.8,1.35)node{\reflectbox{$\ddots$}};
\draw[thick, dotted] (7.5,3.5) -- (10,1) node[below right] {\tiny $S_2(1)$};
\draw[ultra thick,->] (10.5,2.5) -- (11,2);
\draw[thick, dotted] (7,2) -- (9.5,4.5) node[above right] {\tiny $S_1(1)$};
\draw[thick, dotted] (7.5,1.5) -- (10,4);
\draw[thick, dotted] (8,1) -- (10.5,3.5);
\draw (10.25,3.75)node[above right]{$\ddots$};
\draw[thick, dotted] (8.5,0.5) -- (11,3) node[above right] {\tiny $S_1(M)$};
\draw[ultra thick,->] (10,3) -- (10.5,3.5); \draw[ultra thick,->] (9,4) -- (9.5,4.5);
\draw[rotate around={-45:(7.75,3)},fill=white] (7.75,3)rectangle({7.75+3.5*sqrt(2)/2},{3+2.5*sqrt(2)/2})node[midway]{$1$};
\end{scope}
\node[left] at (-1,6) {$u^La_1$};
\node[left] at (-1,4.5) {$\vdots$};
\node[left] at (-1,3) {$u^La_M$};
\node[left] at (-1,1) {$b_N$};
\node[left] at (-1,0.2) {$\vdots$};
\node[left] at (-1,-1) {$b_1$};
\node[below] at (6,-2) {\tiny $m_1(\mu)$};
\node[below] at (5,-2) {\tiny $m_2(\mu)$};
\node[below] at (4,-2) {\tiny $m_3(\mu)$};
\node[below] at (6,7.75) {\tiny $m_1(\mu)$};
\node[below] at (5,7.75) {\tiny $m_2(\mu)$};
\node[below] at (4,7.75) {\tiny $m_3(\mu)$};
\node[below] at (6,2.5) {\tiny $m_1(\lambda)$};
\node[below] at (5,2.5) {\tiny $m_2(\lambda)$};
\node[below] at (4,2.5) {\tiny $m_3(\lambda)$};
\node[below] at (1,7.75) {$\longleftarrow$};
\node[below] at (3,7.75) {$\cdots$};
\node[below] at (0,7.75) {$\infty$};
\node[below] at (1,-2) {$\longleftarrow$};
\node[below] at (3,-2) {$\cdots$};
\node[below] at (0,-2) {$\infty$};
\end{tikzpicture}
\end{multline}
Here, the arrows exiting the right of the grey portion of the lattice enter the $L$th rectangle in the six vertex model, after which it is periodic, in the sense that arrows exiting a rectangle from the top enter the bottom of the next rectangle, as in the quasi-periodic six vertex model. We indicate the identified edges with gray lines. The constraint that $l(\mu)=n$ is replaced with the constraint that $l(\mu)+W_1+\dotsc+W_L=n$, where $W_i$ indicates the number of arrows entering the bottom of the $i$th rectangle of the six vertex model (or alternatively the number of arrows crossing the gray line).

Finally, we take $L\to\infty$. The parameters $u^La_i\to 0$ so the grey portion of the lattice does not allow arrows to exit from the right of any vertex. This forces all arrows in the red portion to travel right, except that possibly they can wrap around the cylinder any number of times. Notice that this decouples the two sides of the diagram, since eventually all arrows will exit the bosonic lattice through the red portion. Then as the deformed boson model has a limit, the six vertex model portion of the diagram must as well, and so it by definition converges to the quasi-periodic six vertex model.

We then see that the expression computes the probability that $S_1$ and $S_2$ are the locations of the exiting arrows in the quasi-periodic six vertex model, and if $W=\sum_{i \geq 1} W_i$ is the number of times arrows wind, then $W+l(\lambda)=n$, where $\lambda$ now has distribution proportional to $u^{|\lambda|}$; the latter fact is a consequence of the now trivialized bosonic portion of \eqref{eq:final-bosonic-lattice}. We write $\chi=l(\lambda)$ and note that it has a $u$-geometric distribution. Finally, the factors introduced from the Yang--Baxter equation and the normalization for the $u$-geometric $\chi$ exactly cancel $\Phi(a,b;0,t,u)$; cf. the equation immediately following \eqref{eq:periodic-macdonald}, with $q=0$.
\end{proof}
\begin{remark}
As noted in Remark \ref{rmk: down right path}, one can modify the proof to allow arbitrary down-right domains for the outgoing arrows, as was done in \cite{BBW16}. In the first step of the proof, one starts with an arbitrary combination of $M$ grey \eqref{eq:black-vertices} and $N$ red \eqref{eq:red-vertices} rows, and after the first iteration of Yang--Baxter moves, a down-right domain emerges from the right of \eqref{eq:first-iteration} (rather than a rectangular domain). After this, the proof proceeds identically.
\end{remark}

\subsection{Stationary periodic six vertex model}
The \emph{periodic six vertex model} can be defined as the $u=1$ case of the quasi-periodic six vertex model of length $L$ (one can also take $L\to\infty$, but certain observables like the number of times arrows wind will diverge). The periodic six vertex model can be viewed as a Markov chain in discrete time, with a single step given by attaching a single six vertex model on an $M\times N$ domain to the end (thus, the states lie in $\{0,1\}^{M+N}$). In particular, it has stationary distributions. 

We will not need the exact form of the stationary distributions in what follows, but since they are simple we give an explanation. It is clear that the number of arrows is preserved. In fact, the stationary distributions can be obtained by conditioning on product Bernoulli distributions, due to the following lemma, which can be proven by a direct calculation.

\begin{lemma}
If we consider a six vertex model vertex with column and row rapidities $a$ and $b$ respectively, and we allow an arrow to enter from the bottom and left with probabilities $a/(1+a)$ and $1/(1+b)$ respectively independently of each other, then arrows exit from the top and right with the same probabilities, independently.
\end{lemma}
By using this lemma, one can propagate a product of Bernoulli random variables associated to columns and rows of parameters $a_i/(1+a_i)$ and $1/(1+b_j)$ respectively, showing that these distributions are stationary. Since the number of arrows are conserved, these distributions are still stationary even after conditioning on the number of arrows, and convex combinations of these give all stationary distributions.

In the finite case, it is clear that simply setting $u=1$ recovers the periodic model, but it is not immediately clear what one obtains in the $u\to 1$ limit of the quasi-periodic six vertex model with $L=\infty$, which is what is related to the periodic Hall--Littlewood process. We show that such a limit can be taken and gives the stationary periodic six vertex model after restricting to any finite portion of the model. It is clear that it suffices to show that the outgoing edges are stationary, since after that the evolution of the arrows simply follows the usual six vertex model rules.

\begin{proposition}
\label{prop: stat 6vm}
As $u\to 1$, $(S_1,S_2)$ converge in distribution to the stationary distribution for the periodic six vertex model conditioned to have $N$ arrows.
\end{proposition}
\begin{proof}
We let $P(z)$ denote the Markov matrix associated to a single $M\times N$ six vertex model with parameter $z$, restricted to the states with exactly $N$ arrows. Then $P(1)$ is an ergodic Markov chain (assuming the parameters are generic). The distribution of the output $(S_1,S_2)$ of the quasi-periodic six vertex model is given by $\mu_0\prod_{i=0}^\infty P(u^i)$, where by convention these matrices will act on row vectors rather than column vectors and these ordered products are taken right to left. Here, $\mu_0$ denotes a deterministic initial condition with arrows entering on the left and not on the bottom. The convergence of this product is implied by Theorem \ref{thm: HL 6vm}. We wish to show that the $u\to 1$ limit of this infinite product exists, and is equal to the stationary distribution for the periodic six vertex model.

In fact, we show that we have equality of the matrices
\begin{equation}
\label{eq: matrix limits}
    \lim_{u\to 1}\prod_{i=0}^{\infty} P(u^i)=\lim_{n\to\infty} P(1)^n,
\end{equation}
and the right hand side is known to give the stationary distribution since $P(1)$ is the transition matrix of an ergodic Markov chain. We let $\|\cdot\|$ denote the operator norm induced by the $l^1$ norm, and recall that this is submultiplicative. Since $P(z)$ is a Markov matrix for all $z$, $\|P(z)\|=1$. We then have
\begin{equation}
\label{eq: stat approx error}
\left\|\prod _{i=0}^\infty P(u^i)-P(1)^n\right\|\leq \left\|\prod _{i=0}^\infty P(u^i)-\prod_{i=n}^\infty P(u^i) P(1)^n\right\|+\left\|\prod_{i=n}^\infty P(u^i) P(1)^n-P(1)^n\right\|.
\end{equation}
The first term satisfies
\begin{equation*}
\left\|\prod _{i=0}^\infty P(u^i)-\prod_{i=n}^\infty P(u^i) P(1)^n\right\|\leq \sum_{i=1}^n \|P(u^i)-P(1)\|,
\end{equation*}
using a telescoping argument, submultiplicativity, and $\|P(z)\|=1$. On the other hand, since $P(1)$ (at least for generic parameters) is the transition matrix of a finite ergodic Markov chain, we have a uniform bound
\begin{equation*}
\|\mu P(1)^n-\nu P(1)^n\|\leq Ce^{-cn}
\end{equation*}
for some constants $c,C>0$, for any distributions $\mu,\nu$. This implies that
\begin{equation*}
\left\|\prod_{i=n}^\infty P(u^i) P(1)^n-P(1)^n\right\|\leq Ce^{-cn},
\end{equation*}
and in particular decays uniformly in $u$. Thus, taking $n$ a function of $u$ growing slowly enough as $u\to 1$ so that the first error term in \eqref{eq: stat approx error} decays (which is possible since $P(z)$ is continuous in $z$ away from a discrete set of poles), we have that the right hand side of \eqref{eq: stat approx error} goes to $0$ as $u\to 1$ and $n\to\infty$. This shows \eqref{eq: matrix limits}.
\end{proof}

\begin{remark}
This relationship between the stationary periodic six vertex model and the periodic Hall--Littlewood process is not completely satisfactory, since the observables $S_1$ and $S_2$ are completely determined by the stationary distribution itself, which has a simple structure. Of more interest is the number of times arrows wind in a finite periodic six vertex model started from stationarity (or any other initial condition). Such problems have been studied for particle systems \cite{BL16,L18,BL21,BL19}, polymers \cite{GK23b}, and the KPZ equation \cite{DGK23, GK23}. It would be very interesting if such an observable could be studied via the periodic Hall--Littlewood measure, but our correspondence falls short of this. 
\end{remark}

\subsection{An identity of symmetric functions}
Using the vertex model representation of the periodic Hall--Littlewood process, we derive here an identity between an infinite sum over a product of two skew Hall--Littlewood polynomials and a Macdonald polynomial indexed by a rectangular partition. The left hand side of this identity essentially takes the form of the distribution function of $\lambda_1$ in the periodic Hall--Littlewood process, modulo some extra multiplicative factors for which we have no immediate probabilistic interpretation. Nevertheless the identity is worth mentioning here, as it provides a direct $t$-generalization of Proposition \ref{prop: vtx model q-whit}.

We first give an {\it uncolored} partition function representation of Macdonald polynomials whose indexing partition has rectangular shape.

\begin{theorem}
\label{thm: macdonald vtx model}
Fix two integers $n,M \geq 1$. The Macdonald polynomial $P_{n^M}(x_1,\dotsc, x_N;q,t)$ is given by
\begin{equation}
\label{eq:rectangle-P-colorblind}
\begin{split}
\frac{(t;t)_M}{(q;q)_{n}}
P_{n^M}(x_1,\dotsc,x_N;q,t)
=
    \sum_{\mu:\mu_1\leq n}q^{|\mu|} \times 
    \begin{tikzpicture}[scale=0.8,baseline=(current bounding box.center),>=stealth]
\foreach\x in {0,...,6}{
\draw[lgray,line width=10pt] (\x,2) -- (\x,7);
}
\foreach\y in {3,...,6}{
\draw[lgray,thick] (-1,\y) -- (7,\y);
}
\node[right] at (7,3) {$0$};
\node[right] at (7,4.5) {$\vdots$};
\node[right] at (7,6) {$0$};
\node[left] at (-1.5,6) {$x_N \rightarrow$};
\node[left] at (-1.5,4.5) {$\vdots$};
\node[left] at (-1.5,3) {$x_1 \rightarrow$};
\node[left] at (-1,6) {$0$};
\node[left] at (-1,4.5) {$\vdots$};
\node[left] at (-1,3) {$0$};
\node[below] at (6,7.75) {\tiny $M$};
\node[below] at (4,7.75) {\tiny $m_2(\mu)$};
\node[below] at (4,1.75) {\tiny $m_2(\mu)$};
\node[below] at (2,7.75) {$\cdots$};
\node[below] at (0,7.75) {\tiny $m_n(\mu)$};
\node[below] at (6,1.75) {\tiny $0$};
\node[below] at (2,1.75) {$\cdots$};
\node[below] at (0,1.75) {\tiny $M{+}m_n(\mu)$};
\node[below] at (5,7.75) {\tiny $m_1(\mu)$};
\node[below] at (5,1.75) {\tiny $m_1(\mu)$};
\end{tikzpicture}
\end{split}
\end{equation}
with $\mu = 1^{m_1(\mu)} 2^{m_2(\mu)} \dots$ as previously, and where the above partition function makes use of the vertex weights \eqref{eq:black-vertices}.
\end{theorem}
\begin{proof}
We begin with a vertex model representation of the {\it non-symmetric Macdonald polynomials}, obtained in \cite{BW22}. We shall apply to this formula a symmetrization procedure analogous to the one performed in \cite{ABW21} (which studied a fermionic vertex model rather than the bosonic model we consider now); in particular, our proof is similar to that of Theorem 15.1.2 in \cite{ABW21}.

Since the notation and conventions in \cite{BW22} differ slightly from ours, we first introduce the vertex model studied there. We define a colored vertex model with weights given by\footnote{This model is a direct transcription of equation (3.7) of \cite{BW22}.}
\begin{equation}
\label{eq:colored conj-weights}
\begin{array}{ccc}
\begin{tikzpicture}[scale=0.8,>=stealth]
\draw[lblue,ultra thick] (-1,0) -- (1,0);
\draw[lblue,line width=10pt] (0,-1) -- (0,1);
\node[below] at (0,-1) {$I$};
\draw[ultra thick,->,rounded corners] (-0.075,-1) -- (-0.075,1);
\draw[ultra thick,->,rounded corners] (0.075,-1) -- (0.075,1);
\node[above] at (0,1) {$I$};
\end{tikzpicture}
\qquad\qquad\qquad
&
\begin{tikzpicture}[scale=0.8,>=stealth]
\draw[lblue,ultra thick] (-1,0) -- (1,0);
\draw[lblue,line width=10pt] (0,-1) -- (0,1);
\node[below] at (0,-1) {$I$};
\draw[ultra thick,->,rounded corners] (-1,0)node[left]{$i$} -- (-0.15,0) -- (-0.15,1);
\draw[ultra thick,->,rounded corners] (0,-1) -- (0,1);
\draw[ultra thick,->,rounded corners] (0.15,-1) -- (0.15,0) -- (1,0)node[right]{$i$};
\node[above] at (0,1) {$I$};
\end{tikzpicture}
\qquad\qquad\qquad
&
\begin{tikzpicture}[scale=0.8,>=stealth]
\draw[lblue,ultra thick] (-1,0) -- (1,0);
\draw[lblue,line width=10pt] (0,-1) -- (0,1);
\node[below] at (0,-1) {$I$};
\draw[ultra thick,->,rounded corners] (-0.075,-1) -- (-0.075,1);
\draw[ultra thick,->,rounded corners] (0.075,-1) -- (0.075,0) -- (1,0)node[right]{$i$};
\node[above] at (0,1) {$I_i^-$};
\end{tikzpicture}
\\
1\qquad\qquad\qquad

&
xt^{I_{[i+1,N]}}\qquad\qquad\qquad

&
x(1-t^{I_i})t^{I_{[i+1,N]}}
\\[15pt]
\begin{tikzpicture}[scale=0.8,>=stealth]
\draw[lblue,ultra thick] (-1,0) -- (1,0);
\draw[lblue,line width=10pt] (0,-1) -- (0,1);
\node[below] at (0,-1) {$I$};
\draw[ultra thick,->,rounded corners] (-1,0)node[left]{$i$} -- (-0.15,0) -- (-0.15,1);
\draw[ultra thick,->,rounded corners] (0,-1) -- (0,1);
\draw[ultra thick,->,rounded corners] (0.15,-1) -- (0.15,1);
\node[above] at (0,1) {$I_i^+$};
\end{tikzpicture}
\qquad\qquad\qquad
&
\begin{tikzpicture}[scale=0.8,>=stealth]
\draw[lblue,ultra thick] (-1,0) -- (1,0);
\draw[lblue,line width=10pt] (0,-1) -- (0,1);
\node[below] at (0,-1) {$I$};
\draw[ultra thick,->,rounded corners] (-1,0)node[left]{$i$} -- (-0.15,0) -- (-0.15,1);
\draw[ultra thick,->,rounded corners] (0,-1) -- (0,1);
\draw[ultra thick,->,rounded corners] (0.15,-1) -- (0.15,0) -- (1,0)node[right]{$j$};
\node[above] at (0,1) {$I_{ij}^{+-}$};
\end{tikzpicture}
\qquad\qquad\qquad
&
\begin{tikzpicture}[scale=0.8,>=stealth]
\draw[lblue,ultra thick] (-1,0) -- (1,0);
\draw[lblue,line width=10pt] (0,-1) -- (0,1);
\node[below] at (0,-1) {$I$};
\draw[ultra thick,->,rounded corners] (-1,0)node[left]{$j$} -- (-0.15,0) -- (-0.15,1);
\draw[ultra thick,->,rounded corners] (0,-1) -- (0,1);
\draw[ultra thick,->,rounded corners] (0.15,-1) -- (0.15,0) -- (1,0)node[right]{$i$};
\node[above] at (0,1) {$I_{ji}^{+-}$};
\end{tikzpicture}
\\
1\qquad\qquad\qquad
&
x(1-t^{I_j})t^{I_{[j+1,N]}}\qquad\qquad\qquad
&
0
\end{array}
\end{equation}
where it is assumed that $1 \leq i < j \leq N$. Here $I=(I_1,\dotsc, I_N)$ is a vector with $I_i$ denoting the number of arrows of color $i$ present at the bottom-incoming edge of the vertex, and $I_{S}=\sum_{i\in S} I_i$ for any set $S \subseteq [1,N]$. Using this model, for any permutation $\rho \in \mathfrak{S}_N$ we define the function
\begin{equation}
\label{eq:f-rho}
\begin{split}
    f^\rho_{(n^M,0^{N-M})}(x_1,\dotsc, x_N)=\begin{tikzpicture}[scale=0.8,baseline=(current bounding box.center),>=stealth]
\foreach\x in {0,...,6}{
\draw[lblue,line width=10pt] (\x,2) -- (\x,7);
}
\foreach\y in {3,...,6}{
\draw[lblue,thick] (-1,\y) -- (7,\y);
}
\node[right] at (7,3) {$0$};
\node[right] at (7,4.5) {$\vdots$};
\node[right] at (7,6) {$0$};
\draw[ultra thick,->] (-1,3) -- (0,3); \draw[ultra thick,->] (-1,4) -- (0,4); \draw[ultra thick,->] (-1,5) -- (0,5);\draw[ultra thick,->] (-1,6) -- (0,6);
\node[left] at (-2,6) {$x_N \rightarrow$};
\node[left] at (-2,4.5) {$\vdots$};
\node[left] at (-2,3) {$x_1 \rightarrow$};
\node[left] at (-1,6) {\tiny $\rho(N)$};
\node[left] at (-1,4.5) {$\vdots$};
\node[left] at (-1,3) {\tiny $\rho(1)$};
\node[above] at (6,7.75) {\tiny $[1,M]$};
\node[above] at (0,7.75) {\tiny $[M{+}1,N]$};
\draw[line width=6pt, ->] (0,6.25) -- (0,7.5);
\draw[line width=6pt, ->] (6,6.25) -- (6,7.5);
\node[below] at (0,2) {$\spadesuit$};
\node[below] at (1,2) {$\clubsuit$};
\end{tikzpicture}
\end{split}
\end{equation}
where color $\rho(i)$ enters the lattice via the left-incoming edge of row $i$ for all $1 \leq i \leq N$. In this partition function there are $n+1$ columns in total, labelled from left to right as $0$ up to $n$. For clarity, in this and all other pictures we shall indicate the $0$th column by the symbol $\spadesuit$ and the $1$st column by $\clubsuit$. Colors $M+1,\dots,N$ exit the lattice via the top-outgoing edge of the $0$th column, while colors $1,\dotsc,M$ can wind arbitrarily many times in any column except the $n$th one (where they exit the lattice), with a fugacity of $q^{n-i}$ associated to such a winding in the $i$th column. In the special case $\rho={\rm id}$, this is a non-symmetric Macdonald polynomial (see Theorem 4.2 of \cite{BW22}) and one can obtain generic $f^{\rho}_{(n^M,0^{N-M})}$ from $f^{\rm id}_{(n^M,0^{N-M})}$ under successive action of generators of the Hecke algebra. 

Notice that in the $0$th column of this partition function, the colors $1,\dotsc, M$ must exit via the same row that they enter, possibly after winding, due to the weight $0$ assigned to a vertex with a larger color entering on the left than exiting on the right; see the bottom-right vertex in \eqref{eq:colored conj-weights}. We can then peel away the weight of the $0$th column in a deterministic way, leading to the equation 
\begin{equation}
\label{eq:f-rho2}
    f^\rho_{(n^M,0^{N-M})}(x_1,\dotsc, x_N)
    =
    C_{\rho}(q,t) \prod_{i=1}^{M} x_{\rho^{-1}(i)} \times 
    \begin{tikzpicture}[scale=0.8,baseline=(current bounding box.center),>=stealth]
\foreach\x in {1,...,6}{
\draw[lblue,line width=10pt] (\x,2) -- (\x,7);
}
\foreach\y in {3,...,6}{
\draw[lblue,thick] (0,\y) -- (7,\y);
}
\node[right] at (7,3) {$0$};
\node[right] at (7,4.5) {$\vdots$};
\node[right] at (7,6) {$0$};
\draw[ultra thick,->] (0,3) -- (1,3); \draw[ultra thick,->] (0,4) -- (1,4); \draw[ultra thick,->] (0,5) -- (1,5);\draw[ultra thick,->] (0,6) -- (1,6);
\node[left] at (-1,6) {$x_N \rightarrow$};
\node[left] at (-1,4.5) {$\vdots$};
\node[left] at (-1,3) {$x_1 \rightarrow$};
\node[left] at (0,6) {\tiny $\rho'(N)$};
\node[left] at (0,4.5) {$\vdots$};
\node[left] at (0,3) {\tiny $\rho'(1)$};
\node[above] at (6,7.75) {\tiny $[1,M]$};
\draw[line width=6pt, ->] (6,6.25) -- (6,7.5);
\node[below] at (1,2) {$\clubsuit$};
\end{tikzpicture}
\end{equation}
where the multiplicative constant $C_{\rho}$ depends on $(q,t)$ but not on the variables $x_i$, and where $\rho'(i) = \rho(i)$ if $1 \leq \rho(i) \leq M$, with $\rho'(i)=0$ otherwise.

We then define the partition function
\begin{equation}
\label{eq:P-colored}
\begin{split}
\mathfrak{P}_{n,M}(x;q,t)
=
\begin{tikzpicture}[scale=0.8,baseline=(current bounding box.center),>=stealth]
\foreach\x in {0,...,6}{
\draw[lblue,line width=10pt] (\x,2) -- (\x,7);
}
\foreach\y in {3,...,6}{
\draw[lblue,thick] (-1,\y) -- (7,\y);
}
\node[right] at (7,3) {$0$};
\node[right] at (7,4.5) {$\vdots$};
\node[right] at (7,6) {$0$};
\node[left] at (-2,6) {$x_N \rightarrow$};
\node[left] at (-2,4.5) {$\vdots$};
\node[left] at (-2,3) {$x_1 \rightarrow$};
\node[left] at (-1,6) {$0$};
\node[left] at (-1,4.5) {$\vdots$};
\node[left] at (-1,3) {$0$};
\node[above] at (6,7.5) {\tiny $[1,M]$};
\node[below] at (0,1.75) {\tiny $[1,M]$};
\draw[line width=6pt, ->] (6,6.25) -- (6,7.5);
\draw[line width=6pt, ->] (0,2-0.25) -- (0,2+0.5);
\node[below] at (0,1) {$\spadesuit$};
\node[below] at (1,1) {$\clubsuit$};
\end{tikzpicture}
\end{split}
\end{equation}
with the same boundary conditions as in \eqref{eq:f-rho}, except that only $M$ colored arrows enter the lattice, and all via the bottom-incoming edge of the $0$th column (winding may occur in any column except the $n$th). In a configuration of the lattice \eqref{eq:P-colored}, the $M$ colored arrows exit the $0$th column via some rows $k_1,\dots,k_M$ respectively; the resulting weight of the $0$th column is equal to $\prod_{i=1}^{M} x_{k_i}$ multiplied by a constant that is again independent of the variables $x_i$. Comparing \eqref{eq:f-rho2} and \eqref{eq:P-colored}, we conclude that
\begin{equation*}
\mathfrak{P}_{n,M}(x;q,t)
=
\sum_{\rho}
D_{\rho}(q,t)
f^\rho_{(n^M,0^{N-M})}(x_1,\dotsc, x_N)
\end{equation*}
for suitable coefficients $D_{\rho}(q,t)$, placing $\mathfrak{P}_{n,M}(x;q,t)$ in the Hecke orbit of the non-symmetric Macdonald polynomial $f^{\rm id}_{(n^M,0^{N-M})}(x_1,\dots,x_N)$. Moreover, by a standard Yang--Baxter argument, $\mathfrak{P}_{n,M}(x;q,t)$ is symmetric in the $x_i$. Up to a multiplicative constant, the only polynomial with these two properties is the symmetric Macdonald polynomial $P_{n^M}(x_1,\dotsc, x_N;q,t)$. We defer for the moment the explicit computation of the constant.

Now we evaluate $\mathfrak{P}_{n,M}(x;q,t)$ a second way. In particular, let $\mathfrak{Q}_{n,M}(x;q,t)$ denote the right hand side of \eqref{eq:rectangle-P-colorblind}, which uses the uncolored weights \eqref{eq:black-vertices}. We claim that
\begin{equation}
\label{eq:PQ-colorblind}
\mathfrak{P}_{n,M}(x;q,t)
=
E(q,t)
\mathfrak{Q}_{n,M}(x;q,t)
\end{equation}
where $E(q,t)$ is another multiplicative constant whose value is for the moment unimportant. The proof of \eqref{eq:PQ-colorblind} is via the following lemma, which is a direct $t$-generalization of Lemma \ref{lem: periodic schur color blind}:

\begin{lemma}
\label{lem:columns}
Fix integers $1 \leq M \leq N$ and two sets of coordinates $S=\{1 \leq s_1 < \cdots < s_M \leq N\}$, $R=\{1 \leq r_1 < \cdots < r_M \leq N\}$. One then has the following equality of single-column partition functions in the two models \eqref{eq:colored conj-weights} and \eqref{eq:black-vertices}:
\begin{equation}
\label{eq:column-colorblind}
\sum_{b_1,\dots,b_N}
\sum_{I=(I_1,\dots,I_M)}
q^{|I|}
\times
\begin{tikzpicture}[scale=0.8,baseline=(current bounding box.center),>=stealth]
\foreach\y in {-1,...,3}{
\draw[lblue,thick] (0.5,\y) -- (1.5,\y);
}
\draw[lblue,line width=10pt] (1,-2) -- (1,4);
\node[below] at (1,-2) {$I$};
\node[above] at (1,4) {$I$};
\node at (-1.5,-1) {$x_1 \rightarrow $};
\node at (-1.5,3) {$x_N \rightarrow $};
\node[left] at (0.5,-1) {$a_1$};
\node[left] at (0.5,0.2) {$\vdots$};
\node[left] at (0.5,1) {$a_i$};
\node[left] at (0.5,2.2) {$\vdots$};
\node[left] at (0.5,3) {$a_N$};
\node[right] at (1.5,-1) {$b_1$};
\node[right] at (1.5,0.2) {$\vdots$};
\node[right] at (1.5,1) {$b_i$};
\node[right] at (1.5,2.2) {$\vdots$};
\node[right] at (1.5,3) {$b_N$};
\end{tikzpicture}
=
\frac{1}{(qt;t)_{M-1}}
\sum_{m} q^m
\times
\begin{tikzpicture}[scale=0.8,baseline=(current bounding box.center),>=stealth]
\foreach\y in {-1,...,3}{
\draw[lgray,thick] (0.5,\y) -- (1.5,\y);
}
\draw[lgray,line width=10pt] (1,-2) -- (1,4);
\node[below] at (1,-2) {$m$};
\node[above] at (1,4) {$m$};
\node at (-1.5,-1) {$x_1 \rightarrow $};
\node at (-1.5,3) {$x_N \rightarrow $};
\node[left] at (0.5,-1) {$\mathbf{1}_{1 \in S}$};
\node[left] at (0.5,0.2) {$\vdots$};
\node[left] at (0.5,1) {$\mathbf{1}_{i \in S}$};
\node[left] at (0.5,2.2) {$\vdots$};
\node[left] at (0.5,3) {$\mathbf{1}_{N \in S}$};
\node[right] at (1.5,-1) {$\mathbf{1}_{1 \in R}$};
\node[right] at (1.5,0.2) {$\vdots$};
\node[right] at (1.5,1) {$\mathbf{1}_{i \in R}$};
\node[right] at (1.5,2.2) {$\vdots$};
\node[right] at (1.5,3) {$\mathbf{1}_{N \in R}$};
\end{tikzpicture}
\end{equation}
where on the left hand side, $(a_1,\dots,a_N)$ is any vector in $\{0,1,\dots,M\}^N$ whose non-zero elements are pairwise distinct and satisfying $a_{s_i} \not= 0$ for all $1 \leq i \leq M$; similarly, $(b_1,\dots,b_N)$ is summed over all vectors in $\{0,1,\dots,M\}^N$ whose non-zero elements are pairwise distinct and satisfying $b_{r_i} \not= 0$ for all $1 \leq i \leq M$. On the right hand side, an arrow enters the column via each horizontal edge $\{s_1<\cdots<s_M\}$; the arrows leave the column via horizontal edges $\{r_1<\cdots<r_M\}$. The index $I$ is summed over all vectors in $\mathbb{Z}_{\geq 0}^M$, and $m$ over all non-negative integers.
\end{lemma}
\begin{proof}
We shall provide a sketch of the proof, omitting purely technical details. We begin by establishing Lemma \ref{lem:columns} in the case that $s_i=r_i=i$ for all $1 \leq i \leq M$. Let us firstly examine the left hand side of \eqref{eq:column-colorblind}. In this case, we see that each color must exit in the row it entered, as a smaller color cannot exit a row that a bigger color enters, due to the vanishing of the bottom-right vertex in \eqref{eq:colored conj-weights}. Thus, the sum over the $(b_1,\dots,b_N)$ collapses to a single term, and each time color $i$ winds, it contributes a factor of $t^{i-1}q$. We can then compute the left hand side as
\begin{equation*}
    \prod _{i=1}^{M} x_i \sum_{I_1,\dotsc, I_M}\prod_{i=1}^{M}\left(t^{i-1}q\right)^{I_{i}}=\prod _{i=1}^{M} x_i \prod_{i=1}^M \frac{1}{1-qt^{i-1}}.
\end{equation*}
On the other hand, it is easily seen that on the right hand side of \eqref{eq:column-colorblind}, the vertex model contributes a factor of $\prod_{i=1}^{M} x_i$ and the winding a factor of $\frac{1}{1-q}$, so the two sides are equal.

More generally, one can prove that the two sides of \eqref{eq:column-colorblind} are equal for $s_i=i$, for all $1 \leq i \leq M$ but with $\{r_1,\dots,r_M\}$ kept arbitrary. This makes use of explicit formulas for the left and right hand sides of \eqref{eq:column-colorblind}; these are provided via Theorem 4.13 of \cite{BW22} and Section 8 of \cite{CdGW15}, respectively.

The final tool which is required is the Yang--Baxter equation for the vertex models on the two sides; see Proposition 3.4 of \cite{BW22} and Theorem 5.1 of \cite{BBW16}. The key is that the same $R$-matrix appears as the intertwiner of both sides of \eqref{eq:column-colorblind}, being simply the $R$-matrix corresponding to the six vertex model weights defined in Figure \ref{fig:vtx wts}.\footnote{On the left hand side, the fact that the (uncolored) six vertex model $R$-matrix acts as an intertwiner is due to the summation over right-outgoing colors, which induces {\it color blindness} of the colored $R$-matrix.} Making use of the Yang--Baxter exchange relations, is then possible to inductively relate the case of arbitrary $\{s_1,\dots,s_M\}$ and $\{r_1,\dots,r_M\}$ to the known case $s_i=i$, for all $1 \leq i \leq M$ and $\{r_1,\dots,r_M\}$ arbitrary.

\end{proof}

\begin{remark}
To the best of our knowledge, Lemma \ref{lem:columns} is the first example of a colour-blindness phenomenon in a system with periodicity; normally such statements require free summation over top and right edges of a lattice, which is not the case in \eqref{eq:column-colorblind}.
\end{remark}

Equipped with the colour-blindness property \eqref{eq:column-colorblind} (and an analogous version that applies to the extremal columns of our partition functions), one is then able to prove \eqref{eq:PQ-colorblind}, starting from the $n$th column of each partition function and working iteratively towards the $0$th column. This establishes that $\mathfrak{Q}_{n,M}(x;q,t)$ is equal to $P_{n^M}(x_1,\dots,x_N;q,t)$ up to a scalar factor.

Finally, we need to address the overall multiplicative constant. Since the Macdonald polynomials are monic with respect to their expansion over the monomial basis, to determine the constant it suffices to compute the coefficient of the leading monomial $\prod_{i=1}^{M} x_i^n$ within the right hand side of \eqref{eq:rectangle-P-colorblind}. The only lattice configuration which contributes to this monomial is the one where arrows exit the $0$th column in the lowest $M$ rows, and then travel rightward until they arrive at the final column, with potential winding in each intermediate column. The windings contribute a factor of $\frac{1}{1-q^j}$ in column $n-j$, for all $1 \leq j \leq n$. Additionally, there is a factor of $(t;t)_M$ coming from the final column. Since the coefficient of $\prod_{i=1}^{M} x_i^n$ in $P_{n^M}(x_1,\dots,x_N;q,t)$ is $1$, we immediately read off the left hand side of \eqref{eq:rectangle-P-colorblind}.

\end{proof}

\begin{proposition}
\label{prop: PHL Macdonald evaluation}
Fix two alphabets $a=(a_1,\dots,a_M)$ and $b=(b_1,\dots,b_N)$, and let $(a,b^{-1})$ denote the combined alphabet $(a_1,\dots,a_M,b^{-1}_1,\dots,b^{-1}_N)$. We have that
\begin{equation}
\label{eq:HL to Macdonald}
    \sum_{\lambda,\mu:\lambda_1\leq n}\frac{(t;t)_{m_n(\mu)}}{(t;t)_{m_n(\lambda)}}u^{|\mu|}P_{\lambda/\mu}(a;0,t)Q_{\lambda/\mu}(b;0,t)=\frac{1}{(u;u)_n}\prod_{j=1}^N b_j^n P_{n^N}(a,b^{-1};u,t),
\end{equation}
where the left hand side features skew Hall--Littlewood polynomials, and the right hand side is a Macdonald polynomial in parameters $(u,t)$. At $t=0$, this reduces to Proposition \ref{prop: vtx model periodic schur} (with $q \mapsto u$, $M \leftrightarrow N$ and $a \leftrightarrow b$).
\end{proposition}

\begin{proof}
The left hand side of \eqref{eq:HL to Macdonald} is equal to
\begin{equation}
\label{eq:hl-macdonald-1}
\sum_{\mu,\lambda:\lambda_1\leq n}\frac{(t;t)_{m_n(\mu)}}{(t;t)_{m_n(\lambda)}}u^{|\mu|}
\times
\begin{tikzpicture}[scale=0.8,baseline=(current bounding box.center),>=stealth]
\foreach\x in {0,...,6}{
\draw[lgray,line width=10pt] (\x,2) -- (\x,7);
}
\foreach\y in {3,...,6}{
\draw[lgray,thick] (-1,\y) -- (7,\y);
}
\foreach\x in {0,...,6}{
\draw[lred,line width=10pt] (\x,-2) -- (\x,2);
}
\foreach\y in {-1,...,1}{
\draw[lred,thick] (-1,\y) -- (7,\y);
}
\draw[ultra thick,->] (-1,-1) -- (0,-1); \draw[ultra thick,->] (-1,0) -- (0,0); \draw[ultra thick,->] (-1,1) -- (0,1);
\draw[ultra thick,->] (6,1) -- (7,1); \draw[ultra thick,->] (6,4) -- (7,4); \draw[ultra thick,->] (6,6) -- (7,6);
\node[right] at (7,-1) {$S_2(1)$}; 
\node[right] at (7,0.2) {$\vdots$}; 
\node[right] at (7,1) {$S_2(N)$};
\node[right] at (7,3) {$S_1(M)$};
\node[right] at (7,4.5) {$\vdots$};
\node[right] at (7,6) {$S_1(1)$};
\node[left] at (-1,6) {$a_1$};
\node[left] at (-1,4.5) {$\vdots$};
\node[left] at (-1,3) {$a_M$};
\node[left] at (-1,1) {$b_N$};
\node[left] at (-1,0.2) {$\vdots$};
\node[left] at (-1,-1) {$b_1$};
\node[below] at (6,-2) {\tiny $m_1(\mu)$};
\node[below] at (5,-2) {\tiny $m_2(\mu)$};
\node[below] at (4,-2) {\tiny $m_3(\mu)$};
\node[below] at (6,7.75) {\tiny $m_1(\mu)$};
\node[below] at (5,7.75) {\tiny $m_2(\mu)$};
\node[below] at (4,7.75) {\tiny $m_3(\mu)$};
\node[below] at (6,2.5) {\tiny $m_1(\lambda)$};
\node[below] at (5,2.5) {\tiny $m_2(\lambda)$};
\node[below] at (4,2.5) {\tiny $m_3(\lambda)$};
\node[below] at (0,2.5) {\tiny $m_n(\lambda)$};
\node[below] at (3,7.75) {$\cdots$};
\node[below] at (0,7.75) {\tiny $m_n(\mu)$};
\node[below] at (3,-2) {$\cdots$};
\node[below] at (0,-2) {\tiny $m_n(\mu)$};
\end{tikzpicture}
\end{equation}
where the lattice consists of only $n$ columns, since $\lambda_1\leq n$, and where $S_1(1),\dots,S_1(M)$ and $S_2(1),\dots,S_2(N)$ are freely summed over $\{0,1\}$. Requiring that the $N$ left-incoming arrows in \eqref{eq:hl-macdonald-1} instead enter via the bottom-incoming edge of the $0$th column, we find that our expression is equal to
\begin{equation}
\label{eq:hl-macdonald-2}
\sum_{\mu,\lambda:\lambda_1\leq n}u^{|\mu|}
\times
\begin{tikzpicture}[scale=0.8,baseline=(current bounding box.center),>=stealth]
\foreach\x in {0,...,6}{
\draw[lgray,line width=10pt] (\x,2) -- (\x,7);
}
\foreach\y in {3,...,6}{
\draw[lgray,thick] (-1,\y) -- (7,\y);
}
\foreach\x in {0,...,6}{
\draw[lred,line width=10pt] (\x,-2) -- (\x,2);
}
\foreach\y in {-1,...,1}{
\draw[lred,thick] (-1,\y) -- (7,\y);
}
\draw[ultra thick,->] (6,1) -- (7,1); \draw[ultra thick,->] (6,4) -- (7,4); \draw[ultra thick,->] (6,6) -- (7,6);
\node[right] at (7,-1) {$S_2(1)$}; 
\node[right] at (7,0.2) {$\vdots$}; 
\node[right] at (7,1) {$S_2(N)$};
\node[right] at (7,3) {$S_1(M)$};
\node[right] at (7,4.5) {$\vdots$};
\node[right] at (7,6) {$S_1(1)$};
\node[left] at (-1,6) {$a_1$};
\node[left] at (-1,4.5) {$\vdots$};
\node[left] at (-1,3) {$a_M$};
\node[left] at (-1,1) {$b_N$};
\node[left] at (-1,0.2) {$\vdots$};
\node[left] at (-1,-1) {$b_1$};
\node[below] at (6,-2) {\tiny $m_1(\mu)$};
\node[below] at (5,-2) {\tiny $m_2(\mu)$};
\node[below] at (4,-2) {\tiny $m_3(\mu)$};
\node[below] at (6,7.75) {\tiny $m_1(\mu)$};
\node[below] at (5,7.75) {\tiny $m_2(\mu)$};
\node[below] at (4,7.75) {\tiny $m_3(\mu)$};
\node[below] at (6,2.5) {\tiny $m_1(\lambda)$};
\node[below] at (5,2.5) {\tiny $m_2(\lambda)$};
\node[below] at (4,2.5) {\tiny $m_3(\lambda)$};
\node[below] at (0,2.5) {\tiny $m_n(\lambda)$};
\node[below] at (3,7.75) {$\cdots$};
\node[below] at (0,7.75) {\tiny $m_n(\mu)$};
\node[below] at (3,-2) {$\cdots$};
\node[below] at (0,-2) {\tiny $N+m_n(\mu)$};
\end{tikzpicture}
\end{equation}
where the removal of the factor $\frac{(t;t)_{m_n(\mu)}}{(t;t)_{m_n(\lambda)}}$ from the summand of \eqref{eq:hl-macdonald-1} is due to $(1-t^k)$ factors that are present in the $0$th column of \eqref{eq:hl-macdonald-1}, but not that of \eqref{eq:hl-macdonald-2}.

Next, we use the fact that the weights of the grey and red portions of the lattice differ by inversion of the rapidities $b_j$ followed by multiplication by $\prod_{j=1}^N b_j^n$. Finally, we add a rightmost column collecting all the arrows into one top-outgoing edge, which changes the partition function by a factor of $(t;t)_N$, and we recover the right hand side of \eqref{eq:HL to Macdonald} after application of Theorem \ref{thm: macdonald vtx model}.
\end{proof}

\begin{remark}
By applying the Macdonald involution $\omega_{0,t}$ to the left hand side of \eqref{eq:HL to Macdonald}, and setting $t \mapsto q$, we map it to an expression that is structurally close to the distribution function of $l(\lambda)$ under the periodic $q$-Whittaker measure (again, there are extra multiplicative factors present in the summand, that prevent this from being an exact correspondence). Application of the involution to the right hand side of \eqref{eq:HL to Macdonald} may be carried out using known integral formulas for the Macdonald polynomials, giving rise to contour integral formulas of the same nature as Theorem \ref{thm: qt sym fn}. 
\end{remark}

\section*{Acknowledgements}
The authors would like to thank Amol Aggarwal and Alexei Borodin for helpful discussions. MW was supported by ARC Future Fellowship FT200100981.

\bibliography{bibliography}{}
\bibliographystyle{abbrvurl}

\end{document}